\documentclass{amsart}
\usepackage{enumitem}
\usepackage{amsmath,amssymb,amsfonts}
\usepackage[all]{xy}
\usepackage{tikz-cd}
\usepackage{bm}

\renewcommand{\S}{\mathcal{S}}
\newcommand{\MS}{\mathcal{MS}}

\newcommand{\I}{\mathrm{I}}

\renewcommand{\P}{\mathcal{P}}
\newcommand{\B}{\mathcal{B}}

\newcommand{\R}{\mathbb{R}}

\newcommand{\CW}{{\mathrm{CW}}}

\newcommand{\simplex}{\bm{\Delta}}
\newcommand{\chains}{\mathrm{C}_\bullet}

\newcommand{\tensor}{\otimes}
\newcommand{\Hom}{\mathrm{Hom}}
\newcommand{\End}{\mathrm{End}}

\newcommand{\x}{\mathbf{x}}

\renewcommand{\1}{\mathbf{1}}

\newtheorem{theorem}{theorem}
\newtheorem{proposition}[theorem]{Proposition}
\newtheorem{lemma}[theorem]{Lemma}
\newtheorem{corollary}[theorem]{Corollary}
\theoremstyle{definition}
\newtheorem{definition}[theorem]{Definition}

\newtheorem{remark}[theorem]{Remark}
\newtheorem{notation}[theorem]{Notation}

\newtheorem{construction}[theorem]{Construction}

\makeatletter
\newcommand*{\eqdef}{\mathrel{\rlap{%
			\raisebox{0.3ex}{$\m@th\cdot$}}%
		\raisebox{-0.3ex}{$\m@th\cdot$}}%
	=}
\makeatother

\usetikzlibrary{arrows,decorations.markings}
\tikzset{myptr/.style={decoration={markings,mark=at position 1 with %
			{\arrow[scale=1.5,>=stealth]{>}}},postaction={decorate}}}

\newsavebox\preproduct
\begin{lrbox}{\preproduct}
	\begin{tikzpicture}[scale=.25]
	\draw (0,0)--(0,-.7);
	\draw (0,0)--(.5,.5) node[above, scale = .5]{$s$};
	\draw (0,0)--(-.5,.5) node[above, scale = .5]{$\scriptstyle 1$-$s$};
	\end{tikzpicture} 
\end{lrbox}
\newcommand{\product}{\usebox\preproduct}

\newsavebox\prenakedproduct
\begin{lrbox}{\prenakedproduct}
	\begin{tikzpicture}[scale=.25]
	\draw (0,0)--(0,-.7);
	\draw (0,0)--(.5,.5) ;
	\draw (0,0)--(-.5,.5);
	\end{tikzpicture} 
\end{lrbox}
\newcommand{\nakedproduct}{\usebox\prenakedproduct}

\newsavebox\precoproduct
	\begin{lrbox}{\precoproduct}
		\begin{tikzpicture}[scale=.25]
		\draw (0,0)--(0,.8);
		\draw (0,0)--(.5,-.5);
		\draw (0,0)--(-.5,-.5);
		\end{tikzpicture}
	\end{lrbox}
\newcommand{\coproduct}{\usebox\precoproduct}

\newsavebox\precoproductcounit
\begin{lrbox}{\precoproductcounit}
	\begin{tikzpicture}[scale=.25]
	\draw (0,0)--(0,.8);
	\draw (0,0)--(.5,-.5);
	\draw (0,0)--(-.5,-.5);
	\draw [fill] (.5,-.5) circle [radius=0.1];
	\draw [fill] (-.5,-.5) circle [radius=0.1];
	\end{tikzpicture}
\end{lrbox}

\newsavebox\preboundary
\begin{lrbox}{\preboundary}
	\begin{tikzpicture}[scale=.3]
	\draw (0,0)--(0,1.3);
	\draw (.5,0)--(.5,1.3);
	\draw [fill] (0,0) circle [radius=0.1];
	\draw (.9,.6)--(1.3,.6);
	\draw (1.7,0)--(1.7,1.3);
	\draw (2.2,0)--(2.2,1.3);
	\draw [fill] (2.2,0) circle [radius=0.1];
	\end{tikzpicture}
\end{lrbox}
\newcommand{\boundary}{
	\usebox\preboundary}

\newsavebox\precoboundary
\begin{lrbox}{\precoboundary}
	\begin{tikzpicture}[scale=.3]
	\draw (0,0)--(0,1.3);
	\draw (.5,0)--(.5,1.3);
	\draw [fill] (0,1.3) circle [radius=0.1];
	\draw (.9,.6)--(1.3,.6);
	\draw (1.7,0)--(1.7,1.3);
	\draw (2.2,0)--(2.2,1.3);
	\draw [fill] (2.2,1.3) circle [radius=0.1];
	\end{tikzpicture}
\end{lrbox}

\newsavebox\precounit
\begin{lrbox}{\precounit}
	\begin{tikzpicture}[scale=.25]
	\draw (0,0)--(0,1.3);
	\draw [fill] (0,0) circle [radius=0.1];
	\end{tikzpicture}
\end{lrbox}
\newcommand{\counit}{
	\usebox\precounit}

\newsavebox\preidentity
\begin{lrbox}{\preidentity}
	\begin{tikzpicture}[scale=.3]
	\draw (0,0)--(0,1.3);
	\end{tikzpicture}
\end{lrbox}
\newcommand{\identity}{
	\usebox\preidentity}

\newsavebox\preunit
\begin{lrbox}{\preunit}
	\begin{tikzpicture}[scale=.3]
	\draw (0,0)--(0,1.3);
	\draw [fill] (0,1.3) circle [radius=0.1];
	\end{tikzpicture}
\end{lrbox}

\newsavebox\preassociativity
\begin{lrbox}{\preassociativity}
	\begin{tikzpicture}[scale=.2]
	\path[draw] (0,0)--(0,1)--(-1,2);
	\draw (0,1)--(1,2);
	\draw (-.5,1.5)--(0,2);
	\draw (1.2,1)--(1.8,1);
	\path[draw] (3,0)--(3,1)--(2,2);
	\draw (3,1)--(4,2);
	\draw (3.5,1.5)--(3,2);	
	\end{tikzpicture}
\end{lrbox}

\newsavebox\precoassociativity
\begin{lrbox}{\precoassociativity}
	\begin{tikzpicture}[scale=.2]
	\path[draw] (0,0)--(0,-1)--(-1,-2);
	\draw (0,-1)--(1,-2);
	\draw (-.5,-1.5)--(0,-2);
	\draw (1.2,-1)--(1.8,-1);
	\path[draw] (3,0)--(3,-1)--(2,-2);
	\draw (3,-1)--(4,-2);
	\draw (3.5,-1.5)--(3,-2);
	\end{tikzpicture}
\end{lrbox}

\newsavebox\preinvolution
\begin{lrbox}{\preinvolution}
	\begin{tikzpicture}[scale=.2]
	\path[draw] (0,0)--(0,.5)--(-.5,1)--(0,1.5)--(0,2);
	\path[draw] (0,.5)--(.5,1)--(0,1.5);
	\end{tikzpicture}
\end{lrbox}
\newcommand{\involution}{
	\usebox\preinvolution}

\newsavebox\preleftcounitality
\begin{lrbox}{\preleftcounitality}
	\begin{tikzpicture}[scale=.3]
	\draw (0,0)--(0,.8);
	\draw (0,0)--(.5,-.5);
	\draw (0,0)--(-.5,-.5);
	\draw [fill] (-.5,-.5) circle [radius=0.1];
	\draw (.7,0)--(1.1,0);
	\path[draw] (1.5,-.5)--(1.5,.8);
	\end{tikzpicture}
\end{lrbox}
\newcommand{\leftcounitality}{
	\usebox\preleftcounitality}

\newsavebox\prerightcounitality
\begin{lrbox}{\prerightcounitality}
	\begin{tikzpicture}[scale=.3]
	\draw (0,0)--(0,.8);
	\draw (0,0)--(.5,-.5);
	\draw (0,0)--(-.5,-.5);
	\draw [fill] (.5,-.5) circle [radius=0.1];
	\draw (-.7,0)--(-1.1,0);
	\path[draw] (-1.5,-.5)--(-1.5,.8);
	\end{tikzpicture}
\end{lrbox}
\newcommand{\rightcounitality}{
	\usebox\prerightcounitality}

\newsavebox\preleftunitality
\begin{lrbox}{\preleftunitality}
	\begin{tikzpicture}[scale=.3]
	\draw (0,0)--(0,-.8);
	\draw (0,0)--(-.5,.5);
	\draw (0,0)--(.5,.5);
	\draw [fill] (.5,.5) circle [radius=0.1];
	\draw (.7,0)--(1.1,0);
	\path[draw] (1.5,.5)--(1.5,-.8);
	\end{tikzpicture}
\end{lrbox}

\newsavebox\prerightunitality
\begin{lrbox}{\prerightunitality}
	\begin{tikzpicture}[scale=.3]
	\draw (0,0)--(0,-.8);
	\draw (0,0)--(-.5,.5);
	\draw (0,0)--(.5,.5);
	\draw [fill] (-.5,.5) circle [radius=0.1];
	\draw (-.7,0)--(-1.1,0);
	\path[draw] (-1.5,.5)--(-1.5,-.8);
	\end{tikzpicture}
\end{lrbox}

\newsavebox\preproductcounit
\begin{lrbox}{\preproductcounit}
	\begin{tikzpicture}[scale=.3]
	\draw (0,0)--(0,-.8);
	\draw (0,0)--(.5,.5);
	\draw (0,0)--(-.5,.5);
	\draw [fill] (0,-.8) circle [radius=0.1];
	\end{tikzpicture}
\end{lrbox}
\newcommand{\productcounit}{
	\usebox\preproductcounit}

\newsavebox\preunitcoproduct
\begin{lrbox}{\preunitcoproduct}
	\begin{tikzpicture}[scale=.3]
	\draw (0,0)--(0,.8);
	\draw (0,0)--(.5,-.5);
	\draw (0,0)--(-.5,-.5);
	\draw [fill] (0,.8) circle [radius=0.1];
	\end{tikzpicture}
\end{lrbox}

\newsavebox\preleibniz
\begin{lrbox}{\preleibniz}
	\begin{tikzpicture}[scale=.245]
	\draw (0,.3)--(0,-.3);
	\draw (0,.3)--(.5,.8);
	\draw (0,.3)--(-.5,.8);
	\draw (0,-.3)--(0,.3);
	\draw (0,-.3)--(.5,-.8);
	\draw (0,-.3)--(-.5,-.8);
	
	\draw (.7,0)--(1.1,0);
	\draw (2.1,.8)--(2.1,.3)--(1.7,0)--(1.7,-.8);
	\draw (2.1,.3)--(2.9,-.3);
	\draw (3.3,.8)--(3.3,0)--(2.9,-.3)--(2.9,-.8);
	
	\draw (3.8,0)--(4.1,0);
	\draw (4.6,.8)--(4.6,0)--(5,-.3)--(5,-.8);
	\draw (5,-.3)--(5.8,.3);
	\draw (5.8,.8)--(5.8,.3)--(6.2,0)--(6.2,-.8);	
	\end{tikzpicture}
\end{lrbox}

\newsavebox\prebialgebra
\begin{lrbox}{\prebialgebra}
	\begin{tikzpicture}[scale=.5]
	\draw (0,.3)--(0,-.3);
	\draw (0,.3)--(.5,.8);
	\draw (0,.3)--(-.5,.8);
	\draw (0,-.3)--(0,.3);
	\draw (0,-.3)--(.5,-.8);
	\draw (0,-.3)--(-.5,-.8);
	
	\draw (.8,0)--(1.2,0);
	
	\draw (2.1,.8)--(2.1,.3)--(1.7,0)--(2.1,-.3)--(2.1,-.8);
	
	\draw (3.1,.8)--(3.1,.3)--(3.5,0)--(3.1,-.3)--(3.1,-.8);

	\draw (2.1,.3)--(3.1,-.3);
	\draw (2.1,-.3)--(2.5,-.06);
	\draw (3.1,.3)--(2.7,.06);	
	\end{tikzpicture}
\end{lrbox}

\newsavebox\precommutativity
\begin{lrbox}{\precommutativity}
	\begin{tikzpicture}[scale=.27]
	\draw (.3,0)--(.3,-1);
	\draw (.3,0)--(.8,.5);
	\draw (.3,0)--(-.2,.5);
	
	\draw (1.2,-.4)--(1.8,-.4);
	
	\draw (3,-.3)--(3,-1);
	\draw (2.5,0)--(3,-.3);
	\draw (3.5,0)--(3,-.3);
	\draw (2.46,0)--(2.9,.18);
	\draw (3.1,.3)--(3.5,.5);
	\draw (3.5,0)--(2.5,.5);
		\end{tikzpicture} 
	\end{lrbox}

\newsavebox\preleftcounithomotopy
\begin{lrbox}{\preleftcounithomotopy}
	\begin{tikzpicture}[scale=.25]
	\draw (0,0)--(0,1.4);
	\draw [fill=white] (0,.7) circle [radius=0.13];
	\node at (.25,1.15)[scale = .5]{$s$};
	\end{tikzpicture}
\end{lrbox}
\newcommand{\leftcounithomotopy}{
	\usebox\preleftcounithomotopy}

\newsavebox\prenakedleftcounithomotopy
\begin{lrbox}{\prenakedleftcounithomotopy}
	\begin{tikzpicture}[scale=.28]
	\draw (0,0)--(0,1.4);
	\draw [fill=white] (0,.7) circle [radius=0.13];
	\end{tikzpicture}
\end{lrbox}
\newcommand{\nakedleftcounithomotopy}{
	\usebox\prenakedleftcounithomotopy}

\usepackage[bookmarks=true, linktocpage=true,
bookmarksnumbered=true, breaklinks=true,
pdfstartview=FitH, hyperfigures=false,
plainpages=false, naturalnames=true,
colorlinks=true, pagebackref=true,
pdfpagelabels]{hyperref}
\hypersetup{
	colorlinks,
	citecolor=blue,
	filecolor=blue,
	linkcolor=blue,
	urlcolor=blue
}

\calclayout

\begin{document}

\title{A finitely presented $E_\infty$-prop II: cellular context}
\author{Anibal M. Medina-Mardones}
\address{Department of Mathematics, University of Notre Dame, Notre Dame, IN 46556, United States}
\address{Max Planck Institute for Mathematics, Bonn, 53111, Germany}
\email{amedinam@nd.edu}
\thanks{Partially supported by Innosuisse grant 32875.1 IP-ICT - 1.}

\begin{abstract}
	We construct, using finitely many generating cell and relations, props in the category of CW-complexes with the property that their associated operads are models for the $E_\infty$-operad. We use one of these to construct a cellular $E_\infty$-bialgebra structure on the interval and derive from it a natural cellular $E_\infty$-coalgebra structure on the geometric realization of a simplicial set which, passing to cellular chains, recovers up to signs the Barratt-Eccles and Surjection coalgebra structures introduced by Berger-Fresse and McClure-Smith. We use another prop, a quotient of the first, to relate our constructions to earlier work of Kaufmann and prove a conjecture of his. This is the second of two papers in a series, the first investigates analogue constructions in the category of chain complexes.
\end{abstract}

\maketitle

\section{Introduction}

This is the second of two papers that, with the exception of section 5.3., can be read independently. In the first \cite{medina2020prop1}, we work over the category of differential graded modules. In this one, we do so over the category of \mbox{$\CW$-complexes} and \mbox{$\CW$-maps}. 

A purposeful construction of model for the $E_\infty$-operad is central in most contexts where commutativity up to coherent homotopies plays a role. No model for the $E_\infty$-operad can be described in terms of finitely many generating cells and relations. However, as demonstrated in this work, passing to a more general setting with multiple inputs and outputs allows to finitely present props whose associated operad is a model for the $E_\infty$-operad. 

We introduce three such props related by quotient morphisms
\begin{equation*}
\tilde{\S} \to \S \to \MS.
\end{equation*}
The cellular chains on the props $\S$ and $\MS$ are isomorphic to the algebraic props introduced in \cite{medina2020prop1}, in particular, the cellular chains of the operad associated to $\MS$ are isomorphic up to signs to the Surjection operad of \cite{mcclure2003multivariable, berger2004combinatorial}. We use $\tilde{\S}$ to provide the interval with an $E_\infty$-bialgebra structure and derive from it a natural $E_\infty$-coalgebra structure on the geometric realization of a simplicial set extending a cellular approximation to the diagonal. We describe how the natural Barratt-Eccles and Surjection coalgebra structures defined in \cite{berger2004combinatorial} and \cite{mcclure2003multivariable} on the normalized chains of simplicial sets are deduced, up to signs, from the $E_\infty$-bialgebra structure on the interval introduced here. Additionally, we prove that $\MS$ is isomorphic to an Arc Surface prop \cite{kaufmann03arc} whose associated operad, introduced by Kaufmann in \cite{kaufmann09dimension}, was conjectured in section 4.4. loc. cit. to have cellular chains isomorphic to the Surjection operad.

Given the concise graphical language it provides, the combinatorial formulation we present in this paper is of independent interest. In a string topology interpretation \cite{tradler07string, kaufmann08frobenious}, we note that the isomorphism of our prop $\MS$ and Kaufmann's stabilized Arc Surface prop induces an action of the cellular chains of $\MS$ on the Hochschild cochains of a normalized semi-simple Frobenius algebra, see Theorem 8 in \cite{kaufmann2008noncommutative}. The coproduct and product in our presentation correspond, via the input-output duality, to the Chas-Sullivan product \cite{chas1999string} and Goresky-Hingston coproduct \cite{goresky09loop} in the formalism of \cite{kaufmann2018detailed}.

We present an overview of the content of this article. In the second section, we review the material on operads and props needed for the rest of the paper, in particular we define the notion of finite presentation of a cellular prop. In the third section, we finitely present the prop $\tilde{\S}$ and compute its homotopy type. In the fourth section, we construct a natural $E_\infty$-coalgebra structure on the geometric realization of simplicial sets from an $\tilde{\S}$-bialgebra structure on the interval. In the final section, we introduce the props $\S$ and $\MS$ and study their relationships to Arc Surface props and the Surjection operad.

\subsection*{Acknowledgments} 

We would like to thank Dennis Sullivan, Bruno Vallette, Ralph Cohen, Stephan Stolz, Ralph Kaufmann, Kathryn Hess, Manuel Rivera and the anonymous referee for their insights, questions, and comments about this project. 

\section{Preliminaries}

We work in the symmetric monoidal category $(\CW,\times,\1)$ of CW-complexes and CW-maps. We denote the interval $[0,1]$ endowed with its usual CW-structure by $\I$.

\subsection{$E_\infty$-operads and $E_\infty$-props}

We say that an operad $\mathcal{O}$ is $\Sigma$\textbf{-free} if the action of $\Sigma_m$ on $\mathcal{O}(m)$ is free for every $m$. A $\Sigma$-\textbf{free resolution} of an operad $\mathcal{O}$ is an operad morphism from a $\Sigma$-free operad to $\mathcal{O}$ inducing a homotopy equivalence in each arity $m$. 

For any $X \in \CW$, there are two types of representations of an operad $\mathcal{O}$ on $X$. They are referred to as $\mathcal{O}$-\textbf{coalgebra} and $\mathcal{O}$-\textbf{algebra} structures and are respectively given by collections of CW-maps
\begin{equation*}
\{\mathcal{O}(m) \times X \to X^m\}_{m \geq 0} \text{ \ and \ } \{\mathcal{O}(m) \times X^m \to X\}_{m \geq 0}
\end{equation*}	
satisfying associativity, equivariance, and unitality relations.

The terminal operad $\1 = \{\1\}_{m\geq0}$ is of particular importance. Its (co)algebras define usual (co)commutative, (co)associative, and (co)unital (co)algebra structures. 

Following May \cite{may2006geometry}, an operad $\mathcal{O}$ is called an $E_\infty$-\textbf{operad} if it is a $\Sigma$-free resolution of the terminal operad and $\mathcal{O}(0)=\1$.

A \textbf{prop} is a strict symmetric monoidal category $\P = (\P, \odot, 0)$ enriched in $\CW$ generated by a single object. For any prop $\P$ with generator $p$ denote the CW-complex $\Hom_\P(p^{\odot n}, p^{\odot m})$ by $\P(n,m)$. The symmetry of the monoidal structure induces commuting right and left actions of $\Sigma_n$ and $\Sigma_m$ on $\P(n,m)$. Therefore, we think of the data of a prop as a $\Sigma$-\textbf{biobject}, i.e., a collection $\P = \big\{\P(n,m)\big\}_{n,m\geq0}$ of CW-complexes with commuting actions of $\Sigma_n$ and $\Sigma_m$, together with three types of maps
\begin{align*}
\circ_h : \P(n_1,m_1) \times \cdots \times \P(n_s,m_s) &\to \P(n_1+\cdots+n_s, m_1+\cdots+m_s), \\
\circ_v : \P(n,k) \times \P(k,m) &\to \P(n,m), \\
\eta : \1 &\to \P(n,n).
\end{align*}
These types of maps are referred to respectively as \textbf{horizontal compositions}, \textbf{vertical compositions}, and \textbf{units}. They are derived respectively from the monoidal product, the categorical composition, and the identity morphisms of $\P$. 

For any $\CW$-complex $X$ there are two types of representations of a prop $\P$ on $X$. They are referred to as $\P$-\textbf{bialgebra} and \textbf{opposite} $\P$-\textbf{bialgebra} structures and are respectively given by collections of CW-maps 
\begin{equation*}
\big\{ \P(n,m) \times X^n \to X^m \big\}_{n,m \geq 0} \text{\ \ and \ } \big\{ \P(n,m) \times X^m \to X^n \big\}_{n,m \geq 0.}
\end{equation*}
satisfying associativity, equivariance, and unitality relations.

Let $U$ be the functor from the category of props to that of operads given by naturally inducing from a prop $\P$ an operad structure on the $\Sigma$-module $U(\P)=\{\P(1,m)\}_{m\geq0}$. Notice that a $\P$-bialgebra (resp. opposite $\P$-algebra) structure on $X$ induces a $U(\P)$-coalgebra (resp. $U(\P)$-algebra) structure on $X$.

Following Boardman and Vogt \cite{boardman2006homotopy}, a prop $\P$ is called an $E_\infty$-\textbf{prop} if $U(\P)$ is an $E_\infty$-operad.

\subsection{Free props and presentations}	As described for example in \cite{fresse2010props}, the \textbf{free prop} $F(\B)$ generated by a \mbox{$\Sigma$-bimodule} $\B$ is constructed using open directed graphs with no directed loops that are enriched with a labeling described next. We think of each directed edge as built from two compatibly directed half-edges. For each vertex $v$ of a directed graph $G$, we have the sets $in(v)$ and $out(v)$ of half-edges that are respectively incoming to and outgoing from $v$. Half-edges that do not belong to $in(v)$ or $out(v)$ for any $v$ are divided into the disjoint sets $in(G)$ and $out(G)$ of incoming and outgoing external half-edges. For any positive integer $n$, let $\overline{n} = \{1,\dots,n\}$ and $\overline{0} = \emptyset$. For any finite set $S$, denote the cardinality of $S$ by $|S|$. The labeling is given by bijections  
\begin{equation*}
\overline{|in(G)|}\to in(G)\hspace*{1cm}\overline{|out(G)|}\to out(G)
\end{equation*}
and
\begin{equation*}
\overline{|in(v)|}\to in(v)\hspace*{1cm}\overline{|out(v)|}\to out(v)
\end{equation*}
for every vertex $v$. We refer to the isomorphism classes of such labeled directed graphs with no directed loops as $(n,m)$\textbf{-graphs}. We consider the right action of $\Sigma_n$ and the left action of $\Sigma_m$ on a $(n,m)$-graph given respectively by permuting the labels of $in(G)$ and $out(G)$. 

The free prop $F(\B)$ is given by all $(n,m)$-graphs which are $\B$-decorated in the following way. To every vertex $v$ of one such $G$, one assigns an element $p \in \B(|in(v)|, |out(v)|)$ and introduces the equivalence relations:  
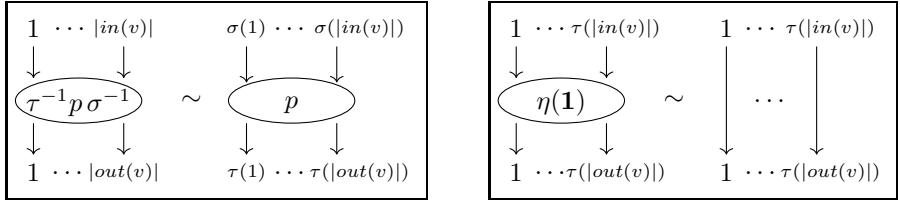
\begin{center}
	\boxed{
		\begin{tikzpicture}[scale=.6]
		\node at (0,2.1) {$1$}; \draw[->] (0,1.7) -- (0,1);
		\node at (.9,2.1) {$\dots\ $};
		\node at (2,2.1) {$\scriptstyle |in(v)|$}; \draw[->] (2,1.7) -- (2,1);
		
		\draw (2.4,.5) arc (0:360:1.4cm and .5cm); \node at (1, .45) {$\tau^{-1} p\,\sigma^{-1}$};
		
		\node at (0,-1.1) {$1$}; \draw[<-](0,-.7) -- (0,0);
		\node at (.8,-1.1) {$\dots$};
		\node at (2,-1.1) {$\scriptstyle|out(v)|$}; \draw[<-] (2,-.7) -- (2,0);
		
		\node at (3.5,.5) {$\sim$}; 
		\end{tikzpicture}\; 
		\begin{tikzpicture}[scale=.6]
		\node at (0,2.1) {$\scriptstyle \sigma(1)$}; \draw[->] (0,1.7) -- (0,.9);
		\node at (1,2.1) {$\dots$};
		\node at (2.5,2.1) {$\scriptstyle\sigma({|in(v)|})$}; \draw[->] (2,1.7) -- (2,.9);
		
		\draw (2.4,.5) arc (0:360:1.4cm and .5cm); \node at (1, .4) {$p$};
		
		\node at (0,-1.1) {$\scriptstyle \tau(1)$}; \draw[<-](0,-.7) -- (0,0);
		\node at (1,-1.1) {$\dots$};
		\node at (2.5,-1.1) {$\scriptstyle\tau(|out(v)|)$}; \draw[<-] (2,-.7) -- (2,0);
		\end{tikzpicture}
	} 
	\qquad
	\boxed{
		\begin{tikzpicture}[scale=.6]
		\node at (0,2.1) {$1$}; \draw[->] (0,1.7) -- (0,1);
		\node at (.8,2.1) {$\dots$};
		\node at (2.2,2.1) {$\scriptstyle\tau(|in(v)|)$}; \draw[->] (2,1.7) -- (2,1);
		
		\draw (2.4,.5) arc (0:360:1.4cm and .5cm); \node at (1, .45) {$\eta(\1)$};
		
		\node at (0,-1.1) {$1$}; \draw[<-](0,-.7) -- (0,0);
		\node at (.8,-1.1) {$\dots$};
		\node at (2.2,-1.1) {$\scriptstyle\tau(|out(v)|)$}; \draw[<-] (2,-.7) -- (2,0);
		
		\node at (3.5,.5) {$\sim$}; 
		\end{tikzpicture}\ \ 
		\begin{tikzpicture}[scale=.6]
		\node at (0,2.1) {$1$}; \draw[->] (0,1.7) -- (0,-.7);
		\node at (2.3,2.1) {$\scriptstyle\tau(|in(v)|)$}; \draw[->] (2,1.7) -- (2,-.7);
		
		\node at (.8,2.1) {$\dots$};
		\node at (1,.5) {$\dots$};
		\node at (.8,-1.1) {$\dots$};
		
		\node at (0,-1.1) {$1$};
		\node at (2.3,-1.1) {$\scriptstyle \tau(|out(v)|)$};
		\end{tikzpicture}
	}
\end{center}
where $p \in \B(n,m)$, $\sigma\in\Sigma_n$, and $\tau\in\Sigma_m$.

For any $\Sigma$-bimodule $\B$, the above construction defines the free prop $F(\B)$ associated to $\B$. It satisfies the following universal property: Let $\iota : \B \to F(\B)$ be the morphism sending an element $p \in \B(n,m)$ to the labeled and decorated $(n,m)$-corolla
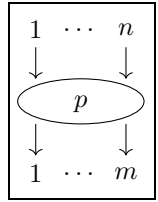
\begin{center}
	\boxed{
		\begin{tikzpicture}[scale=.6]
		\node at (0,2.1) {$1$}; \draw[->] (0,1.7) -- (0,1);
		\node at (1,2.1) {$\dots$};
		\node at (2,2.1) {$n$}; \draw[->] (2,1.7) -- (2,1);
		
		\draw (2.4,.5) arc (0:360:1.4cm and .5cm); \node at (1, .45) {$p$};
		
		\node at (0,-1.1) {$1$}; \draw[<-](0,-.7) -- (0,0);
		\node at (1,-1.1) {$\dots$};
		\node at (2,-1.1) {$m$}; \draw[<-] (2,-.7) -- (2,0);
		
		\end{tikzpicture}}
\end{center}
For any $\Sigma$-bimodule map $\phi : \B \to \P$ where $\P$ is a prop, there exists a unique prop morphism 
\begin{equation*}
F(\phi) : F(\B) \to \P
\end{equation*} 
such that 
\begin{equation*}
\phi = F(\phi) \circ \iota.
\end{equation*}
Furthermore, there is a canonical isomorphism $F(F(\B)) \to F(\B)$ given by regarding graphs containing graphs as graphs.

Given any bisequence of spaces $\big\{ B(n,m) \big\}_{n,m \geq 0}$ the free $\Sigma$-bimodule $B^\Sigma$ is defined by 
\begin{equation*}
B^\Sigma(n,m) = \Sigma_m \times B(n,m) \times \Sigma_n
\end{equation*}
and satisfies the following universal property: Let $\xi : B \to B^\Sigma$ be the bisequence map that crosses with the identity elements in the corresponding symmetric groups. For any bisequence map $\phi : B \to \B$, there exists a unique $\Sigma$-bimodule map $\phi^\Sigma : B^\Sigma \to \B$ such that $\phi = \phi^\Sigma \circ \xi$.

We will now describe what is meant by a \textbf{presentation} $(G,\Phi, R)$ of a prop. 

The first piece of data is a collection $G=\{G_d\}$ of bisequences with each $G_d(n,m)$ a disjoint union of spaces isomorphic to $\I^d$. Each such space is called a \textbf{generating $d$-cell in biarity} $(n,m)$. We denote the bisequence containing their boundaries by $\partial G_d$ and notice that $(\partial G_d)^\Sigma = \partial G_d^\Sigma$. 

The second piece of data $\Phi$ are the \textbf{generating attaching maps}. These are morphisms of $\Sigma$-bimodules
\begin{equation*}
\varphi_d: \partial G_d^\Sigma \to F(G^\Sigma_{d-1}).
\end{equation*}
Let $X_0$ be equal to $F(G^\Sigma_0)$ and for $d > 0$ let $X_d$ be equal to the pushout
\begin{center}
	\begin{tikzcd}
	F(\partial G_d^\Sigma) \arrow[d, >->] \arrow[r, "F(\varphi_d)\ "] &
	F(F(G_{d-1}^\Sigma)) \cong F(G_{d-1}^\Sigma) \arrow[r] & 
	X_{d-1} \arrow[d] \\
	F(G_d^\Sigma) \arrow[rr] & & X_d
	\end{tikzcd}
\end{center}
The limit of this process $X$ is endowed with the induced prop structure.

The third piece of data is a bisequence $R$ of subcomplexes of $X$ called the \textbf{relations}. Denote by $\langle R \rangle$ the smallest sub-$\Sigma$-bimodule in $X$ containing $R$ and closed under compositions. We say that the triple $(G,\Phi,R)$ is a presentation of the prop $X/\langle R\rangle$.

\subsection{Immersion convention} Graphs immersed in the plane will be used to represent labeled directed graphs with no directed loops, the convention we will follow is that the direction is given from top to bottom and the labeling from left to right. For example,

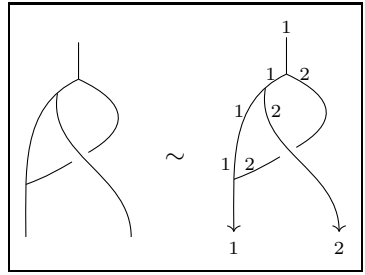
\begin{center}
	\boxed{
		\begin{tikzpicture}[scale=.7]
		\draw (1,3.7) to (1,3); 
		
		\draw (1,3) to [out=205, in=90] (0,0);
		
		\draw [shorten >= 0cm] (.6,2.73) to [out=-100, in=90] (2,0);
		
		\draw [shorten >= .15cm] (1,3) to [out=-25, in=30, distance=1.1cm] (1,1.5);
		\draw [shorten <= .1cm] (1,1.5) to [out=210, in=20] (0,1);
		
		\node at (1,3.9){};
		\node at (0,-.32){};
		\node at (2,-.32){};
		
		\node at (3,1.5){$\sim$\ \ \ };
		\end{tikzpicture}
		\begin{tikzpicture}[scale=.7]
		\draw (1,3.7) to (1,3); 
		
		\draw [->](1,3) to [out=205, in=90] (0,0);
		
		\draw [shorten >= 0cm,->] (.6,2.73) to [out=-100, in=90] (2,0);
		
		\draw [shorten >= .15cm] (1,3) to [out=-25, in=30, distance=1.1cm] (1,1.5);
		\draw [shorten <= .1cm] (1,1.5) to [out=210, in=20] (0,1);
		
		\node at (1,3.9){$\scriptstyle 1$};
		
		\node at (.7,3){$\scriptstyle 1$};
		\node at (1.35,3){$\scriptstyle 2$};
		
		\node at (.1,2.3){$\scriptstyle 1$};
		\node at (.8,2.3){$\scriptstyle 2$};
		
		\node at (-.15,1.3){$\scriptstyle 1$};
		\node at (.3,1.3){$\scriptstyle 2$};
		
		\node at (0,-.3){$\scriptstyle 1$};
		\node at (2,-.3){$\scriptstyle 2$};
		\end{tikzpicture}
	}
\end{center}

\section{The prop $\tilde{\S}$ }
In this section we define the prop $\tilde{\S}$ via a finite presentation and show it is an $E_\infty$-prop.

\begin{definition}\label{Prop S}
	Let $\tilde{\S}$ be the prop generated by 
	\begin{equation*}
	\counit \in \tilde{\S}_0(1,0) \quad \coproduct \in \tilde{\S}_0(1,2) \quad \product \in \tilde{\S}_1(2,1) \quad \leftcounithomotopy \in \tilde{\S}_1(1,1)
	\end{equation*} 
	with generating attaching maps
	\begin{equation*}
	\begin{tikzpicture}[scale=.25]
	\draw (0,0)--(0,-.7);
	\draw (0,0)--(-.5,.5) node[above, scale = .5]{$\scriptstyle 1$};
	\draw (0,0)--(.5,.5) node[above, scale = .5]{$\scriptstyle 0$};)
	\end{tikzpicture}
	=\
	\begin{tikzpicture}[scale=.25]
	\draw (0,0)--(0,1.3);
	\draw [fill] (.7,.1) circle [radius=0.1];
	\draw (.7,0)--(.7,1.3);
	\end{tikzpicture}
	\qquad
	\begin{tikzpicture}[scale=.25]
	\draw (0,0)--(0,-.7);
	\draw (0,0)--(.5,.5) node[above, scale = .5]{$\scriptstyle 1$};
	\draw (0,0)--(-.5,.5) node[above, scale = .5]{$\scriptstyle 0$};)
	\end{tikzpicture}
	=\
	\begin{tikzpicture}[scale=.25]
	\draw (0,0)--(0,1.3);
	\draw [fill] (0,.1) circle [radius=0.1];
	\draw (.7,0)--(.7,1.3);
	\end{tikzpicture}
	\qquad \text{and} \qquad
	\begin{tikzpicture}[scale=.25]
	\draw (0,0)--(0,1.4);
	\draw [fill=white] (0,.7) circle [radius=0.13];
	\node at (.25,1.15)[scale = .5]{$\ 1$};
	\end{tikzpicture}
	=\
	\begin{tikzpicture}[scale=.25]
	\draw (0,0)--(0,.8);
	\draw (0,0)--(.5,-.5);
	\draw (0,0)--(-.5,-.5);
	\draw [fill] (-.5,-.5) circle [radius=0.1];
	\end{tikzpicture}
	\qquad
	\begin{tikzpicture}[scale=.25]
	\draw (0,0)--(0,1.4);
	\draw [fill=white] (0,.7) circle [radius=0.13];
	\node at (.25,1.15)[scale = .5]{$\ 0$};
	\end{tikzpicture}
	=\
	\begin{tikzpicture}[scale=.25]
	\draw (0,-.5)--(0,.9);
	\end{tikzpicture}
	\end{equation*}
	and restricted by the relations 
	\begin{equation*}
	\begin{tikzpicture}[scale=.25]
	\draw (0,0)--(0,.8);
	\draw (0,0)--(.5,-.5);
	\draw (0,0)--(-.5,-.5);
	\draw [fill] (-.5,-.5) circle [radius=0.1];
	\draw [fill] (.5,-.5) circle [radius=0.1];
	\end{tikzpicture}
	=
	\begin{tikzpicture}[scale=.25]
	\draw (0,0)--(0,1.3);
	\draw [fill] (0,0) circle [radius=0.1];
	\end{tikzpicture}
	\qquad
	\begin{tikzpicture}[scale=.25]
	\draw (0,0)--(0,-.7);
	\draw (0,0)--(.5,.5) node[above, scale = .5]{$s$};
	\draw (0,0)--(-.5,.5) node[above, scale = .5]{$\scriptstyle 1$-$s$};)
	\draw [fill] (0,-.65) circle [radius=0.1];
	\end{tikzpicture}
	\!=\
	\begin{tikzpicture}[scale=.25]
	\draw (0,0)--(0,1.3);
	\draw [fill] (0,0) circle [radius=0.1];
	\draw (.7,0)--(.7,1.3);
	\draw [fill] (.7,0) circle [radius=0.1];
	\end{tikzpicture}
	\qquad\ \ 
	\begin{tikzpicture}[scale=.25]
	\draw (0,0)--(0,1.4);
	\draw [fill=white] (0,.7) circle [radius=0.13];
	\node at (.25,1.15)[scale = .5]{$s$};
	\draw [fill] (0,0) circle [radius=0.1];
	\end{tikzpicture}
	=\,
	\begin{tikzpicture}[scale=.25]
	\draw (0,0)--(0,1.3);
	\draw [fill] (0,0) circle [radius=0.1];
	\end{tikzpicture}
	\ .
	\end{equation*}
\end{definition}

Recall that $\1$ stands for the terminal CW-complex, i.e., a single 0-cell.
\begin{lemma} \label{lemma: homotopy type of S tilde}
	Let 
	\begin{equation*}
	\overline{\1}(n,m) =
	\begin{cases} 
	\quad \1 & \text{ if } n>0 \\
	\quad \emptyset & \text{ if } n=0
	\end{cases}	
	\end{equation*}
	endowed with the trivial prop structure. The unique map $\tilde{\S}\to\overline{\1}$ is a homotopy equivalence.
\end{lemma}

\begin{proof} 			
	For $n=0$, we notice that $\tilde{\S}(0,m) = \emptyset = \overline{\1}(0,m)$. For $n>0$ and $m\geq0$ we start by showing that the CW-complexes $\tilde{\S}(n,m)$ and $\tilde{\S}(n,m+1)$ are homotopy equivalent. Consider the collection of maps $\{i:\tilde{\S}(n,m)\to\tilde{\S}(n,m+1)\}$ described by the following diagram\\
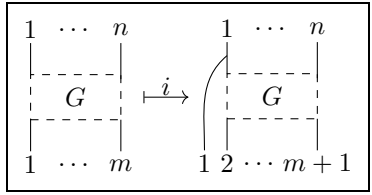
	\begin{equation*}
	\boxed{\begin{tikzpicture}[scale=.6]
		\node at (0,2) {$1$}; \draw (0,1.7) -- (0,1);
		\node at (1,2) {$\dots$};
		\node at (2,2) {$n$}; \draw (2,1.7) -- (2,.9);
		
		\draw [dashed] (0,0) rectangle (2,1); \node at (1, .5) {$G$};
		
		\node at (0,-1) {$1$}; \draw (0,-.7) -- (0,0);
		\node at (1,-1) {$\dots$};
		\node at (2,-1) {$m$}; \draw (2,-.7) -- (2,0);
		
		\node at (3,.75) {$i$}; \draw [|->] (2.5,.5) -- (3.5,.5);
		\end{tikzpicture}
		\begin{tikzpicture}[scale=.6]
		\node at (0,2) {$1$}; \draw (0,1.7) -- (0,1);
		\node at (1,2) {$\dots$};
		\node at (2,2) {$n$}; \draw (2,1.7) -- (2,.9);
		
		\draw [dashed] (0,0) rectangle (2,1); \node at (1, .5) {$G$};
		
		\node at (-.5,-1) {$1$}; \draw (-.5,-.7)  to[out=90,in=-140]  (0,1.4);
		\node at (0,-1) {$2$}; \draw (0,-.7) -- (0,0);
		\node at (.75,-1) {$\dots$};
		\node at (2,-1) {$m+1$}; \draw (2,-.7) -- (2,0);
		\end{tikzpicture}}
	\end{equation*}\\
	Consider also the collection of maps $\{r:\tilde{\S}(n,m+1)\to\tilde{\S}(n,m)\}$ described by\\
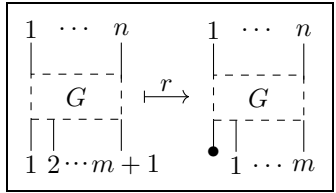
	\begin{equation*}
	\boxed{\begin{tikzpicture}[scale=.6]
		\node at (0,2) {$1$}; \draw (0,1.7) -- (0,1);
		\node at (1,2) {$\dots$};
		\node at (2,2) {$n$}; \draw (2,1.7) -- (2,.9);
		
		\draw [dashed] (0,0) rectangle (2,1); \node at (1, .5) {$G$};
		
		\draw (0,-.7) -- (0,0);
		\node at (0,-1) {$1$};
		\node at (.5,-1) {$2$}; \draw (.5,-.7) -- (.5,0);
		\node at (1,-1) {$\cdot\!\cdot\!\cdot$};
		\node at (2.1,-1) {$m+1$}; \draw (2,-.7) -- (2,0);
		
		\node at (3,.75) {$r$}; \draw [|->] (2.5,.5) -- (3.5,.5);
		\end{tikzpicture}
		\ 
		\begin{tikzpicture}[scale=.6]
		\node at (0,2) {$1$}; \draw (0,1.7) -- (0,1);
		\node at (1,2) {$\dots$};
		\node at (2,2) {$n$}; \draw (2,1.7) -- (2,.9);
		
		\draw [dashed] (0,0) rectangle (2,1); \node at (1, .5) {$G$};
		
		\node at(0,-.7) {$\bullet$}; \draw (0,-.7) -- (0,0);
		\node at (.5,-1) {$1$}; \draw (.5,-.7) -- (.5,0);
		\node at (1.25,-1) {$\dots$};
		\node at (2,-1) {$m$}; \draw (2,-.7) -- (2,0);
		\end{tikzpicture}}
	\end{equation*}\\
	The diagram below shows that $r\circ i$ is homotopic to the identity\\
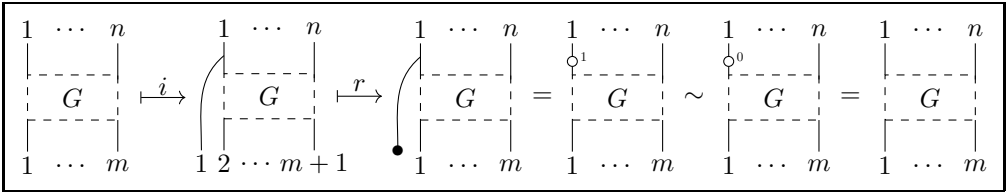
	\begin{equation*}
	\boxed{\begin{tikzpicture}[scale=.6]
		\node at (0,2) {$1$}; \draw (0,1.7) -- (0,1);
		\node at (1,2) {$\dots$};
		\node at (2,2) {$n$}; \draw (2,1.7) -- (2,.9);
		
		\draw [dashed] (0,0) rectangle (2,1); \node at (1, .5) {$G$};
		
		\node at (0,-1) {$1$}; \draw (0,-.7) -- (0,0);
		\node at (1,-1) {$\dots$};
		\node at (2,-1) {$m$}; \draw (2,-.7) -- (2,0);
		
		\node at (3,.75) {$i$}; \draw [|->] (2.5,.5) -- (3.5,.5);
		\end{tikzpicture}
		\begin{tikzpicture}[scale=.6]
		\node at (0,2) {$1$}; \draw (0,1.7) -- (0,1);
		\node at (1,2) {$\dots$};
		\node at (2,2) {$n$}; \draw (2,1.7) -- (2,.9);
		
		\draw [dashed] (0,0) rectangle (2,1); \node at (1, .5) {$G$};
		
		\node at (-.5,-1) {$1$}; \draw (-.5,-.7)  to[out=90,in=-140]  (0,1.4);
		\node at (0,-1) {$2$}; \draw (0,-.7) -- (0,0);
		\node at (.75,-1) {$\dots$};
		\node at (2,-1) {$m+1$}; \draw (2,-.7) -- (2,0);
		
		\node at (3,.75) {$r$}; \draw [|->] (2.5,.5) -- (3.5,.5);
		\end{tikzpicture}
		\begin{tikzpicture}[scale=.6]
		\node at (0,2) {$1$}; \draw (0,1.7) -- (0,1);
		\node at (1,2) {$\dots$};
		\node at (2,2) {$n$}; \draw (2,1.7) -- (2,.9);
		
		\draw [dashed] (0,0) rectangle (2,1); \node at (1, .5) {$G$};
		
		\node at (-.5,-.7) {$\bullet$}; \draw (-.5,-.7)  to[out=90,in=-140]  (0,1.4);
		\node at (0,-1) {$1$}; \draw (0,-.7) -- (0,0);
		\node at (1,-1) {$\dots$};
		\node at (2,-1) {$m$}; \draw (2,-.7) -- (2,0);
		\node at (2.7,.5) {$=$}; 
		\end{tikzpicture}
		\!
		\begin{tikzpicture}[scale=.6]
		\node at (0,2) {$1$}; \draw (0,1.7) -- (0,1);
		\draw [fill=white] (0,1.3) circle [radius=0.11];
		\node at (.25,1.4)[scale = .5]{$1$};
		\node at (1,2) {$\dots$};
		\node at (2,2) {$n$}; \draw (2,1.7) -- (2,.9);
		
		\draw [dashed] (0,0) rectangle (2,1); \node at (1, .5) {$G$};
		
		\node at (0,-1) {$1$}; \draw (0,-.7) -- (0,0);
		\node at (1,-1) {$\dots$};
		\node at (2,-1) {$m$}; \draw (2,-.7) -- (2,0);
		\node at (2.7,.5) {$\sim$}; 
		\end{tikzpicture}
		\begin{tikzpicture}[scale=.6]
		\node at (0,2) {$1$}; \draw (0,1.7) -- (0,1);
		\draw [fill=white] (0,1.3) circle [radius=0.11];
		\node at (.25,1.4)[scale = .5]{$0$};
		\node at (1,2) {$\dots$};
		\node at (2,2) {$n$}; \draw (2,1.7) -- (2,.9);
		
		\draw [dashed] (0,0) rectangle (2,1); \node at (1, .5) {$G$};
		
		\node at (0,-1) {$1$}; \draw (0,-.7) -- (0,0);
		\node at (1,-1) {$\dots$};
		\node at (2,-1) {$m$}; \draw (2,-.7) -- (2,0);
		\node at (2.7,.5) {$=$}; 
		\end{tikzpicture}
		\begin{tikzpicture}[scale=.6]
		\node at (0,2) {$1$}; \draw (0,1.7) -- (0,1);
		\node at (1,2) {$\dots$};
		\node at (2,2) {$n$}; \draw (2,1.7) -- (2,.9);
		
		\draw [dashed] (0,0) rectangle (2,1); \node at (1, .5) {$G$};
		
		\node at (0,-1) {$1$}; \draw (0,-.7) -- (0,0);
		\node at (1,-1) {$\dots$};
		\node at (2,-1) {$m$}; \draw (2,-.7) -- (2,0);
		\end{tikzpicture}
	}
	\end{equation*}\\
	Let us compute diagrammatically the composition $i\circ r$\\
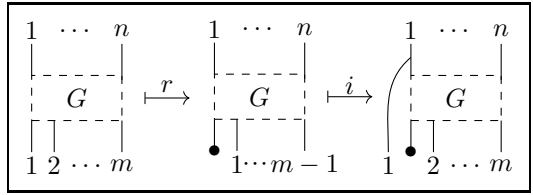
	\begin{equation*}
	\boxed{\begin{tikzpicture}[scale=.6]
		\node at (0,2) {$1$}; \draw (0,1.7) -- (0,1);
		\node at (1,2) {$\dots$};
		\node at (2,2) {$n$}; \draw (2,1.7) -- (2,.9);
		
		\draw [dashed] (0,0) rectangle (2,1); \node at (1, .5) {$G$};
		
		\draw (0,-.7) -- (0,0);
		\node at (0,-1) {$1$};
		\node at (.5,-1) {$2$}; \draw (.5,-.7) -- (.5,0);
		\node at (1.25,-1) {$\dots$};
		\node at (2,-1) {$m$}; \draw (2,-.7) -- (2,0);
		
		\node at (3,.75) {$r$}; \draw [|->] (2.5,.5) -- (3.5,.5);
		\end{tikzpicture}
		\   
		\begin{tikzpicture}[scale=.6]
		\node at (0,2) {$1$}; \draw (0,1.7) -- (0,1);
		\node at (1,2) {$\dots$};
		\node at (2,2) {$n$}; \draw (2,1.7) -- (2,.9);
		
		\draw [dashed] (0,0) rectangle (2,1); \node at (1, .5) {$G$};
		
		\node at(0,-.7) {$\bullet$}; \draw (0,-.7) -- (0,0);
		\node at (.5,-1) {$1$}; \draw (.5,-.7) -- (.5,0);
		\node at (.9,-1) {$\cdot\!\cdot\!\cdot$};
		\node at (2,-1) {$m-1$}; \draw (2,-.7) -- (2,0);
		
		\node at (3,.75) {$i$}; \draw [|->] (2.5,.5) -- (3.5,.5);
		\end{tikzpicture}
		\begin{tikzpicture}[scale=.6]
		\node at (0,2) {$1$}; \draw (0,1.7) -- (0,1);
		\node at (1,2) {$\dots$};
		\node at (2,2) {$n$}; \draw (2,1.7) -- (2,.9);
		
		\draw [dashed] (0,0) rectangle (2,1); \node at (1, .5) {$G$};
		
		\node at (-.5,-1) {$1$}; \draw (-.5,-.7)  to[out=90,in=-140]  (0,1.4);
		\node at (0,-.7) {$\bullet$}; \draw (0,-.7) -- (0,0);
		\node at (.5,-1) {$2$}; \draw (.5,-.7) -- (.5,0);
		\node at (1.25,-1) {$\dots$};
		\node at (2,-1) {$m$}; \draw (2,-.7) -- (2,0);
		\end{tikzpicture}
	}
	\end{equation*}\\
	The composition $i \circ r$ is homotopic to the identity since
	\begin{equation*}
	\boxed
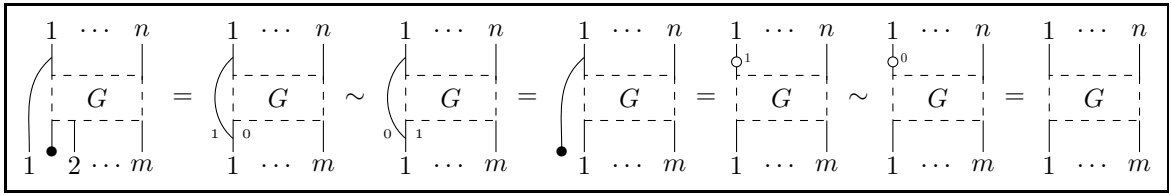
{
		\begin{tikzpicture}[scale=.6]
		\node at (0,2) {$1$}; \draw (0,1.7) -- (0,1);
		\node at (1,2) {$\dots$};
		\node at (2,2) {$n$}; \draw (2,1.7) -- (2,.9);
		
		\draw [dashed] (0,0) rectangle (2,1); \node at (1, .5) {$G$};
		
		\node at (-.5,-1) {$1$}; \draw (-.5,-.7)  to[out=90,in=-140]  (0,1.4);
		\node at (0,-.7) {$\bullet$}; \draw (0,-.7) -- (0,0);
		\node at (.5,-1) {$2$}; \draw (.5,-.7) -- (.5,0);
		\node at (1.25,-1) {$\dots$};
		\node at (2,-1) {$m$}; \draw (2,-.7) -- (2,0);
		\node at (2.9,.5) {$=$}; 
		\end{tikzpicture}
		\begin{tikzpicture}[scale=.6]
		\node at (0,2) {$1$}; \draw (0,1.7) -- (0,1);
		\node at (1,2) {$\dots$};
		\node at (2,2) {$n$}; \draw (2,1.7) -- (2,.9);
		
		\draw [dashed] (0,0) rectangle (2,1); \node at (1, .5) {$G$};
		
		\draw (0,-.4)  to[out=140,in=-140]  (0,1.4);
		\node at (0,-1) {$1$}; \draw (0,-.7) -- (0,0);
		\node at (1,-1) {$\dots$};
		\node at (2,-1) {$m$}; \draw (2,-.7) -- (2,0);
		
		\node at (-.4,-.3){$\scriptscriptstyle{1}$};
		\node at (.3,-.3){$\scriptscriptstyle{0}$};
		
		\node at (2.7,.5) {$\sim$}; 
		\end{tikzpicture}
		\begin{tikzpicture}[scale=.6]
		\node at (0,2) {$1$}; \draw (0,1.7) -- (0,1);
		\node at (1,2) {$\dots$};
		\node at (2,2) {$n$}; \draw (2,1.7) -- (2,.9);
		
		\draw [dashed] (0,0) rectangle (2,1); \node at (1, .5) {$G$};
		
		\draw (0,-.4)  to[out=140,in=-140]  (0,1.4);
		\node at (0,-1) {$1$}; \draw (0,-.7) -- (0,0);
		\node at (1,-1) {$\dots$};
		\node at (2,-1) {$m$}; \draw (2,-.7) -- (2,0);
		
		\node at (-.4,-.3){$\scriptscriptstyle{0}$};
		\node at (.3,-.3){$\scriptscriptstyle{1}$};
		
		\node at (2.7,.5) {$=$}; 
		\end{tikzpicture}
		\begin{tikzpicture}[scale=.6]
		\node at (0,2) {$1$}; \draw (0,1.7) -- (0,1);
		\node at (1,2) {$\dots$};
		\node at (2,2) {$n$}; \draw (2,1.7) -- (2,.9);
		
		\draw [dashed] (0,0) rectangle (2,1); \node at (1, .5) {$G$};
		
		\node at (-.5,-.7) {$\bullet$}; \draw (-.5,-.7)  to[out=90,in=-140]  (0,1.4);
		\node at (0,-1) {$1$}; \draw (0,-.7) -- (0,0);
		\node at (1,-1) {$\dots$};
		\node at (2,-1) {$m$}; \draw (2,-.7) -- (2,0);
		\node at (2.7,.5) {$=$}; 
		\end{tikzpicture}
		\!
		\begin{tikzpicture}[scale=.6]
		\node at (0,2) {$1$}; \draw (0,1.7) -- (0,1);
		\draw [fill=white] (0,1.3) circle [radius=0.11];
		\node at (.25,1.4)[scale = .5]{$1$};
		\node at (1,2) {$\dots$};
		\node at (2,2) {$n$}; \draw (2,1.7) -- (2,.9);
		
		\draw [dashed] (0,0) rectangle (2,1); \node at (1, .5) {$G$};
		
		\node at (0,-1) {$1$}; \draw (0,-.7) -- (0,0);
		\node at (1,-1) {$\dots$};
		\node at (2,-1) {$m$}; \draw (2,-.7) -- (2,0);
		\node at (2.7,.5) {$\sim$}; 
		\end{tikzpicture}
		\begin{tikzpicture}[scale=.6]
		\node at (0,2) {$1$}; \draw (0,1.7) -- (0,1);
		\draw [fill=white] (0,1.3) circle [radius=0.11];
		\node at (.25,1.4)[scale = .5]{$0$};
		\node at (1,2) {$\dots$};
		\node at (2,2) {$n$}; \draw (2,1.7) -- (2,.9);
		
		\draw [dashed] (0,0) rectangle (2,1); \node at (1, .5) {$G$};
		
		\node at (0,-1) {$1$}; \draw (0,-.7) -- (0,0);
		\node at (1,-1) {$\dots$};
		\node at (2,-1) {$m$}; \draw (2,-.7) -- (2,0);
		\node at (2.7,.5) {$=$}; 
		\end{tikzpicture}
		\begin{tikzpicture}[scale=.6]
		\node at (0,2) {$1$}; \draw (0,1.7) -- (0,1);
		\node at (1,2) {$\dots$};
		\node at (2,2) {$n$}; \draw (2,1.7) -- (2,.9);
		
		\draw [dashed] (0,0) rectangle (2,1); \node at (1, .5) {$G$};
		
		\node at (0,-1) {$1$}; \draw (0,-.7) -- (0,0);
		\node at (1,-1) {$\dots$};
		\node at (2,-1) {$m$}; \draw (2,-.7) -- (2,0);
		\end{tikzpicture}
	}\end{equation*}\\
	These computations show that $i$ and $r$ are homotopy inverses, and the relations imposed on $\tilde{\S}$ imply that $\tilde{\S}(n,0)$ contains only the class of \vspace*{-7pt}
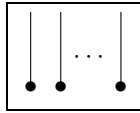
	\begin{center}
		\boxed{\begin{tikzpicture}[scale=.4]
			\node at (0,0){$\bullet$}; \draw (0,0) -- (0,2.5);
			\node at (1,0){$\bullet$}; \draw (1,0) -- (1,2.5);
			\node at (3,0){$\bullet$}; \draw (3,0) -- (3,2.5);
			\node at (2,1){$\cdots$}; 
			\end{tikzpicture}}
	\end{center}\vspace*{-7pt}
	We conclude that each $\tilde{\S}(n,m)$ is contractible for $n > 0$.
\end{proof}

\begin{theorem} \label{sigma-free resolution}
	The prop $\tilde{\S}$ is an $E_\infty$-prop.
\end{theorem}

\begin{proof}
	Since by construction the action of $\Sigma_m$ on $U(\tilde{\S})(m)=\tilde{\S}(1,m)$ is free, the theorem follows from the previous lemma.
\end{proof}

\begin{remark}
	Notice that the operad obtained by restricting to $\{S(n,1)\}$ is not $\Sigma$-free. For example, 
	\begin{center}
		\boxed
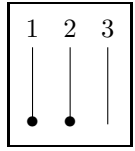
{\begin{tikzpicture}[scale=.5]
			\node at (0,0){$\bullet$}; \draw (0,0) -- (0,2); \node at (0,2.5){$1$};
			\node at (1,0){$\bullet$}; \draw (1,0) -- (1,2); \node at (1,2.5){$2$};
			\draw (2,-.05) -- (2,2); \node at (2,2.5){$3$};
			\end{tikzpicture}}
	\end{center}
	in $\tilde{\S}(3,1)$ is fixed by the transposition $(1,2)$.
\end{remark}

\section{Cellular $E_\infty$-coalgebra on simplicial sets}

In this section we derive from an $\tilde{\S}$-bialgebra structure on the interval a natural \mbox{$U(\tilde{\S})$-coalgebra} structures on the geometric realization of simplicial sets.

\begin{definition} \label{definition: simplex category}
	Let us denote the singleton $\{0\}$ by $\simplex^0$ and the interval $[0,1] \subset \R$ by $\simplex^1$ or $\I$. For $d \geq 1$ let  
	\begin{equation*}
	\simplex^d = \big\{ (x_1, \dots, x_d) \in \I^d\ | \ x_1 \leq \cdots \leq x_d \}.
	\end{equation*}
	For $i = 0,\dots,d+1$ the \textbf{coface maps} $\delta_i:\simplex^{d} \to \simplex^{d+1}$ and \textbf{codegeneracy maps} $\sigma_i : \simplex^{d+1} \to \simplex^{d}$ are respectively defined by
	\begin{equation*}
	\delta_i(x_1, \dots, x_d) = 
	\begin{cases}
	(0,x_1, \dots, x_d)    & i = 0, \\
	(x_1, \dots, x_i, x_i, \dots, x_d)    & 0 < i < d+1, \\
	(x_1, \dots, x_d, 1)    & i = d+1,
	\end{cases}
	\end{equation*}
	and
	\begin{equation*}
	\sigma_i(x_1, \dots, x_{d+1}) = (x_1, \dots, \widehat{x}_i, \dots, x_{d+1}).
	\end{equation*}
	
	We give the spaces $\simplex^d$ the coarser CW-structure making coface and codegeneracy maps into CW-maps. With respect to this CW-structure an element $(x_1,\dots,x_d)$ belongs to the $k$-skeleton of $\simplex^d$ if an only if the cardinality of $\{x_i\; |\; x_i \neq 0,1\}$ is less than or equal to $k$.
	
	The \textbf{simplex category} is the subcategory of CW-complexes with objects $\simplex^d$ and morphisms generated by coface and codegeneracy maps.
\end{definition}

\begin{figure} [h]
	\begin{center}
		\begin{tikzpicture}[scale=4]
		\tikzstyle{vertex}=[minimum size=14pt,inner sep=0pt, minimum size=.3cm]
		\tikzstyle{edge} = [draw,thick,-,black]
		
		\node[vertex] (v0) at (0,0) {(0,0,0)};
		\node[vertex] (v1) at (0,1) {(0,0,1)};
		\node[vertex] (v2) at (1,0) {(1,0,0)};
		\node[vertex] (v3) at (1,1) {(1,0,1)};
		\node[vertex] (v4) at (0.23, 0.4) {(0,1,0)};
		\node[vertex] (v5) at (0.23,1.4) {(0,1,1)};
		\node[vertex] (v6) at (1.23,0.4) {(1,1,0)};
		\node[vertex] (v7) at (1.23,1.4) {(1,1,1)};
		
		\draw (v1) -- (v3) --(v2) ;
		\draw[dashed] (v0) -- (v4) -- (v5) ;
		\draw (v7) -- (v3);
		\draw[dashed] (v4) -- (v6) ;
		
		\draw[very thick] (v0) -- (v1) -- (v5) -- (v7);
		\draw[very thick] (v1) -- (v7);
		\draw[dashed, very thick] (v0) -- (v5);
		\draw[dashed, very thick] (v0) -- (v7);
		\draw (v0) -- (v2);
		\draw (v2) -- (v6);
		\draw (v6) -- (v7);
		\end{tikzpicture}
	\end{center}
	\caption{The simplex $\simplex^3$ drawn with think lines.}
\end{figure}

\begin{definition} \label{definition: action on the interval}
	For arbitrary $x,y \in \simplex^1$ and $s \in \I$ define:
	\begin{enumerate}
		\item the diagonal approximation $\Delta : \simplex^1 \to \simplex^1 \times \simplex^1$ by
		\begin{equation*}
		\Delta(x) = 
		\begin{cases}
		(0,2x)   & \text{if }\ x \leq 1/2, \\
		(2x-1,1) & \text{if }\ x \geq 1/2,
		\end{cases}
		\end{equation*}
		\item the join $\psi : \I \times \simplex^1 \times \simplex^1 \to \simplex^1$ by
		\begin{equation*}
		\psi_s(x,y) = sx + (1-s)y,
		\end{equation*}
		\item the counit homotopy $\phi : \I \times \simplex^1 \to \simplex^1$
		\begin{equation*}
		\phi_s(x) =
		\begin{cases}
		\ \frac{2x}{2-s}  & \text{if }  x \leq \frac{2-s}{2}, \\
		\quad 1 & \text{if } x \geq \frac{2-s}{2},
		\end{cases} 
		\end{equation*}
		\item and the terminal map $\varepsilon : \simplex^1 \to \{0\}$.
	\end{enumerate}	
\end{definition}

\begin{lemma} \label{lemma: action on standard simplices}
	The maps given by
	\begin{align*} 
	&\Phi(\coproduct\,, (x_1,\dots,x_d)) = \big( (\pi_1\Delta(x_1), \dots, \pi_1\Delta(x_d)), (\pi_2\Delta(x_1), \dots, \pi_2\Delta(x_d)) \big), \\
	& \Phi(\ \counit\ , (x_1,\dots,x_d)) = (\varepsilon(x_1), \dots, \varepsilon(x_d)), \\
	& \Phi( \product\!,\, (x_1,\dots,x_d),\, (y_1,\dots,y_d)) = \big( \psi_s(x_1,y_1), \dots, \psi_s(x_d,y_d) \big), \\
	& \Phi( \ \leftcounithomotopy ,\, (x_1,\dots,x_d)) = (\phi_s(x_1), \dots, \phi_s(x_d)),
	\end{align*}
	define a $\tilde{\S}$-bialgebra structure on $\simplex^d$.
\end{lemma}	

\begin{proof} 
	For any $s \in \I$ and $x \in \simplex^1$ each of the functions
	\begin{equation*}
	\pi_1\Delta,\, \pi_2\Delta,\, \psi_s(x, -),\, \psi(-,x),\, \phi_s : \simplex^1 \to \simplex^1
	\end{equation*}
	is order preserving, so the maps above are well defined. By counting the number of distinct coordinates of $(x_1,\dots,x_d)$ that are not equal to $0$ or $1$ before and after applying the maps above we can verify they are cellular. To check these maps define a $\tilde{\S}$-structure we need to verify they satisfy the identities coming from the attaching maps and relations on the generating cells of $\tilde{\S}$. In what follows we use the isomorphisms $\simplex^0 \times \simplex^d \cong \simplex^d \cong \simplex^d \times \simplex^0$ with no further comment. For $\bm{x} = (x_1,\dots,x_n)$ and $\bm{y} = (y_1,\dots,y_n)$ we have\\
	\textit{Attaching maps}:
	\begin{align*}
	\ &\Phi( 
	\begin{tikzpicture}[scale=.25]
	\draw (0,0)--(0,-.7);
	\draw (0,0)--(.5,.5) node[above, scale = .5]{$\scriptstyle 1$};
	\draw (0,0)--(-.5,.5) node[above, scale = .5]{$\scriptstyle 0$};)
	\end{tikzpicture}
	\!,\, \bm x,\, \bm y ) =
	\bm y \cong
	\Phi(\, 
	\begin{tikzpicture}[scale=.25]
	\draw (0,0)--(0,1.3);
	\draw [fill] (0,.1) circle [radius=0.1];
	\draw (.7,0)--(.7,1.3);
	\end{tikzpicture} 
	\, ,\, \bm x,\, \bm y )
	\\
	\ &\Phi( 
	\begin{tikzpicture}[scale=.25]
	\draw (0,0)--(0,-.7);
	\draw (0,0)--(.5,.5) node[above, scale = .5]{$\scriptstyle 0$};
	\draw (0,0)--(-.5,.5) node[above, scale = .5]{$\scriptstyle 1$};)
	\end{tikzpicture}
	\!,\, \bm x,\, \bm y ) =
	\bm x \cong
	\Phi(\, 
	\begin{tikzpicture}[scale=.25]
	\draw (0,0)--(0,1.3);
	\draw [fill] (.7,.1) circle [radius=0.1];
	\draw (.7,0)--(.7,1.3);
	\end{tikzpicture} 
	\, ,\, \bm x,\, \bm y )
	\\
	\ &\Phi(\, 
	\begin{tikzpicture}[scale=.25]
	\draw (0,0)--(0,1.4);
	\draw [fill=white] (0,.7) circle [radius=0.13];
	\node at (.25,1.15)[scale = .5]{$\ 0$};
	\end{tikzpicture}
	\!,\, \bm x) =
	\bm x = 		
	\Phi(\ 
	\begin{tikzpicture}[scale=.25]
	\draw (0,-.5)--(0,.8);
	\end{tikzpicture}
	\ ,\, \bm x )
	\\
	\ & \Phi(\, 
	\begin{tikzpicture}[scale=.25]
	\draw (0,0)--(0,1.4);
	\draw [fill=white] (0,.7) circle [radius=0.13];
	\node at (.25,1.15)[scale = .5]{$\ 1$};
	\end{tikzpicture}
	\!,\, \bm x) =
	\left(\begin{cases}
	2x_1 & \text{if } x_1 \leq \frac{1}{2} \\
	\ 1  & \text{if } x_1 \geq \frac{1}{2}
	\end{cases},
	\dots,
	\begin{cases}
	2x_d  & \text{if } x_d \leq \frac{1}{2} \\
	\ 1 & \text{if } x_d \geq \frac{1}{2}
	\end{cases}
	\right)
	\, \cong \Phi(\, 
	\begin{tikzpicture}[scale=.25]
	\draw (0,0)--(0,.8);
	\draw (0,0)--(.5,-.5);
	\draw (0,0)--(-.5,-.5);
	\draw [fill] (-.5,-.5) circle [radius=0.1];
	\end{tikzpicture}\,,\, \bm x).
	\end{align*}
	\textit{Relations}:
	\begin{align*}
	\ & \Phi( 
	\begin{tikzpicture}[scale=.25]
	\draw (0,0)--(0,.8);
	\draw (0,0)--(.5,-.5);
	\draw (0,0)--(-.5,-.5);
	\draw [fill] (-.5,-.5) circle [radius=0.1];
	\draw [fill] (.5,-.5) circle [radius=0.1];
	\end{tikzpicture}
	\, ,\, \bm x) \cong 0 \cong
	\Phi(\, 
	\begin{tikzpicture}[scale=.25]
	\draw (0,0)--(0,1.3);
	\draw [fill] (0,0) circle [radius=0.1];
	\end{tikzpicture} 
	\, ,\, \bm x )
	\\
	\ & \Phi( 
	\begin{tikzpicture}[scale=.25]
	\draw (0,0)--(0,-.7);
	\draw (0,0)--(.5,.5) node[above, scale = .5]{$s$};
	\draw (0,0)--(-.5,.5) node[above, scale = .5]{$\scriptstyle 1$-$s$};)
	\draw [fill] (0,-.65) circle [radius=0.1];
	\end{tikzpicture}
	\!,\, \bm x,\, \bm y ) \cong 0 \cong
	\Phi(\, 
	\begin{tikzpicture}[scale=.25]
	\draw (0,0)--(0,1.3);
	\draw [fill] (0,0) circle [radius=0.1];
	\draw (.7,0)--(.7,1.3);
	\draw [fill] (.7,0) circle [radius=0.1];
	\end{tikzpicture} 
	\, ,\, \bm x,\, \bm y )
	\\
	\ & \Phi(\, 
	\begin{tikzpicture}[scale=.25]
	\draw (0,0)--(0,1.4);
	\draw [fill=white] (0,.7) circle [radius=0.13];
	\node at (.25,1.15)[scale = .5]{$s$};
	\draw [fill] (0,0) circle [radius=0.1];
	\end{tikzpicture}
	, \bm x) \cong 0 \cong
	\Phi(\, 
	\begin{tikzpicture}[scale=.25]
	\draw (0,0)--(0,1.3);
	\draw [fill] (0,0) circle [radius=0.1];
	\end{tikzpicture} 
	\, ,\, \bm x ).
	\end{align*}	
	Finally we need to verify naturality. Given the coordinate-wise nature of the $\tilde{\S}$-bialgebra structure, naturality with respect to codegeneracy maps and coface maps $\delta_i : \simplex^d \to \simplex^{d+1}$ for $0 < i < d$ is immediate. Using
	\begin{equation*}
	\begin{split}
	\Phi(\coproduct\,, \delta_0 x) & = 
	\big( (\pi_1\Delta(0), \pi_1\Delta(x)), (\pi_2\Delta(0), \pi_2\Delta(x)) \big) \\ & = 
	\big( (0, \pi_1\Delta(x)), (0, \pi_2\Delta(x)) \big) \\ & =
	\delta_0 (\pi_1\Delta(x), \pi_2\Delta(x)) \\ & =
	\delta_0 \Phi(\coproduct\,, x)
	\end{split}
	\end{equation*}
	and
	\begin{align*}
	\Phi(\product, \delta_0 x, \delta_0 y) & =  (\psi_s(0,0), \psi_s(x,y)) &
	\Phi(\,\leftcounithomotopy, \delta_0 x) & = (\phi_s(0), \phi_s(x)) \\ 
	& = (0, \Phi(\product, x, y) ) & & = (0, \Phi(\,\leftcounithomotopy, x) ) \\
	& = \delta_0 \Phi(\product, x, y) & & = \delta_0 \Phi(\,\leftcounithomotopy, x)
	\end{align*}
	we can deduce the naturality of $\delta_0$. The naturality of $\delta_d$ is derived analogously.
\end{proof}

\begin{definition} 
	A \textbf{simplicial set} $\Gamma$ is a contravariant functor from the simplex category to the category of sets. Denote $\Gamma(\simplex^d)$ by $\Gamma_d$. Its \textbf{geometric realization} is the CW-complex 
	\begin{equation*}
	|\Gamma| = \coprod_{d \geq 0} \Gamma_d \times \bm \Delta^d \Big/ \sim		
	\end{equation*}
	where $\tau^*(\gamma) \times \x \sim a \times \tau(\x)$ for any $\tau \in \Hom(\bm \Delta^d, \bm \Delta^{e})$, $\gamma \in \Gamma_{e}$ and $\x \in \bm \Delta^d$. 
\end{definition}

\begin{theorem} \label{theorem: action on simplicial sets}
	Let $\Gamma$ be a simplicial set. A natural $U(\tilde{\S})$-coalgebra structure is defined on its geometric realization by 
	\begin{equation*}
	U(\Phi)\big( g, (\gamma, \x) \big) = \big( (\gamma, \pi_1\Phi(g,\x)),\, \dots\, ,\, (\gamma, \pi_m\Phi(g,\x)) \big)
	\end{equation*}
	where $g \in U(\tilde{\S})(m)$, $\gamma \in \Gamma_d$ and $\x \in \bm \Delta^d$.
\end{theorem}

\begin{proof}
	To verify $U(\Phi) : U(\tilde{\S})(m) \times |\Gamma| \to |\Gamma|^m$ is a well defined map, consider two representatives $\tau^*(\gamma) \times \x$ and $\gamma \times \tau(\x)$. We have
	\begin{equation*}
	\big( \tau^*(\gamma), \pi_1\Phi(g,\x),\ \dots\ ,\ \tau^*(\gamma), \pi_m\Phi(g,\x) \big)
	\end{equation*}
	is equivalent in $|\Gamma|^m$ to
	\begin{equation*}
	\big( \gamma, \tau\big( \pi_1 \Phi(g, \x) \big),\ \dots\ ,\ \gamma, \tau\big( \pi_m\Phi(g,\x) \big) \big)
	\end{equation*}
	which equals
	\begin{equation*}
	\big( \gamma, \pi_1 \Phi(g,\tau(\x)),\ \dots\ ,\ \gamma, \pi_m \Phi(g,\tau(\x)) \big) 
	\end{equation*}	
	since, by naturality,
	\begin{equation*}
	\tau \big( \pi_i \Phi(g,\x) \big) = \pi_i \tau^m \Phi(g,\x) = \pi_i \Phi(g,\tau(\x)).
	\end{equation*}
	The equivariance of $U(\Phi)$ and the identity
	\begin{equation*}
	U(\Phi)\big( h,\, U(\Phi) ( g, (\gamma, \x) ) \big) = U(\Phi)\big( h \circ g, (\gamma, \x) \big)
	\end{equation*}
	follow from those satisfied by $\Phi$.
\end{proof}

\begin{remark}
	It is not the case that $|\Gamma|$ carries an $\tilde{\S}$-bialgebra structure. For example, if $|\Gamma| = \{x,y\}$ then $\Phi(\product,\, x,y)$ is not well defined.
\end{remark}

\begin{remark} \label{remark: explicit action of cellular generators on standard simplices}
	For any $d \geq 0$ and $j \geq 0$ let $[j] = \big( [j]_1, \dots, [j]_d \big) \in \bm \Delta^d$ be given by
	\begin{equation*}
	[j]_k = 
	\begin{cases}
	0 & k \leq d-j \\
	1 & k > d-j.
	\end{cases}
	\end{equation*}
	For any finite set of integers $\{v_0, \dots, v_k\}$, denote the convex closure of $[v_0], \dots, [v_k]$ by $[v_0, \dots, v_k]$. These subsets correspond to the cells of $\bm \Delta^d$ and we have
	\begin{equation*}
	\begin{split}
	& \Phi(\coproduct\,, [v_0,\dots,v_k]) = \coprod_{i=0}^d [v_0,\dots,v_i] \times [v_i,\dots,v_k] \\
	& \Phi(\, \counit\,, [v_0,\dots,v_k]) = [0] \\
	\coprod_{s\in\I} & \Phi(\product,\ [v_0,\dots,v_j],\ [v_{j+1},\dots,v_{k}]) = [v_0,\dots, v_j, v_{j+1}, \dots,v_{k}] \\
	\coprod_{s\in\I} & \Phi(\,\leftcounithomotopy, [v_0,\dots,v_k]) = [v_0,\dots,v_k].
	\end{split}
	\end{equation*}
\end{remark}

\begin{remark}
	We conjecture that a natural cellular action of $U(\tilde{\S})$ can be defined on the geometric realization of cubical sets with or without connections. For cubical chains, an action of the prop obtained by applying cellular chains to $U(\tilde{\S})$ is defined in \cite{medina2020odd}. 
\end{remark}

\section{The prop $\MS$}

In this section we introduce the finitely presented $E_\infty$-prop $\MS$. We provide a description of $\MS$ in terms of oriented surfaces with a weighted 1-skeleton. We show that the operad associate to $\MS$ is isomorphic to Kaufmann's Arc Surface model for the \mbox{$E_\infty$-operad} \cite{kaufmann09dimension}, and that its cellular chains are isomorphic, up to signs, to the Surjection operad \cite{mcclure2003multivariable, berger2004combinatorial}. This two identifications combine to verify a conjecture of Kaufmann.

\subsection{Edge-weights and definition}

\begin{definition} \label{Prop Su}
	Let $\S$ be the prop generated by 
	\begin{equation*}
	\counit \in \S(1,0)_0 \quad \coproduct \in \S(1,2)_0 \quad \product \in \S(2,1)_1
	\end{equation*} 
	with generating attaching maps
	\begin{equation*}
	\begin{tikzpicture}[scale=.25]
	\draw (0,0)--(0,-.7);
	\draw (0,0)--(-.5,.5) node[above, scale = .5]{$\scriptstyle 1$};
	\draw (0,0)--(.5,.5) node[above, scale = .5]{$\scriptstyle 0$};)
	\end{tikzpicture}
	=\
	\begin{tikzpicture}[scale=.25]
	\draw (0,0)--(0,1.3);
	\draw [fill] (.7,.1) circle [radius=0.1];
	\draw (.7,0)--(.7,1.3);
	\end{tikzpicture}
	\qquad
	\begin{tikzpicture}[scale=.25]
	\draw (0,0)--(0,-.7);
	\draw (0,0)--(.5,.5) node[above, scale = .5]{$\scriptstyle 1$};
	\draw (0,0)--(-.5,.5) node[above, scale = .5]{$\scriptstyle 0$};)
	\end{tikzpicture}
	=\
	\begin{tikzpicture}[scale=.25]
	\draw (0,0)--(0,1.4);
	\draw [fill] (0,.1) circle [radius=0.1];
	\draw (.7,0)--(.7,1.3);
	\end{tikzpicture}
	\end{equation*}
	and restricted by the relations 
	\begin{equation*}
	\begin{tikzpicture}[scale=.25]
	\draw (0,0)--(0,.8);
	\draw (0,0)--(.5,-.5);
	\draw (0,0)--(-.5,-.5);
	\draw [fill] (-.5,-.5) circle [radius=0.1];
	\end{tikzpicture}
	=\
	\begin{tikzpicture}[scale=.25]
	\draw (0,0)--(0,1.5);
	\end{tikzpicture}
	\; = 
	\begin{tikzpicture}[scale=.25]
	\draw (0,0)--(0,.8);
	\draw (0,0)--(.5,-.5);
	\draw (0,0)--(-.5,-.5);
	\draw [fill] (.5,-.5) circle [radius=0.1];
	\end{tikzpicture}
	\qquad
	\begin{tikzpicture}[scale=.25]
	\draw (0,0)--(0,-.7);
	\draw (0,0)--(.5,.5) node[above, scale = .5]{$s$};
	\draw (0,0)--(-.5,.5) node[above, scale = .5]{$\scriptstyle 1$-$s$};)
	\draw [fill] (0,-.65) circle [radius=0.1];
	\end{tikzpicture}
	=\
	\begin{tikzpicture}[scale=.25]
	\draw (0,0)--(0,1.3);
	\draw [fill] (0,0) circle [radius=0.1];
	\draw (.7,0)--(.7,1.3);
	\draw [fill] (.7,0) circle [radius=0.1];
	\end{tikzpicture}\,.
	\end{equation*}
\end{definition}	

\begin{remark}
	This prop is a strictly counital version of $\tilde{\S}$ and it receives a quotient map from it. We notice that a simplified version of the proof given for Lemma \ref{lemma: homotopy type of S tilde} shows this is an $E_\infty$-prop.
\end{remark}

We define alternative coordinates on the prop $\S$.	Consider a $d$-cell in $\S(n,m)$ and an ($n,m$)-graph $G$ supporting it. The \textbf{edge-weight coordinates} of this cell are given by the assignment of a non-negative real number to each edge of $G$, referred to as its \textbf{weight}, satisfying the following conditions:\footnote{We say an edge belongs to $in(G),\ out(G),\ in(v)$ or $out(v)$ if one of its two half-edges does.} 
\begin{enumerate}
	\item Edges of the form \counit\ have weight $0$.
	\item Edges in $out(G)$ have weight $1$.
	\item For every vertex $v$ of $G$, the sum of the edge-weights in $in(v)$ and in $out(v)$ are the same.
\end{enumerate}
Edge-weight coordinates are well defined as we can see from:
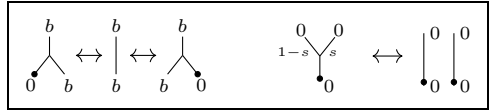
\begin{center}
	\boxed{
		\begin{tikzpicture}[scale=.5]
		\draw (0,0)--(0,.5);
		\node at (0,.8){$\scriptstyle b$};
		\draw (0,0)--(.4,-.5); 
		\draw (0,0)--(-.4,-.5);
		\node at (-.5,-.8){$\scriptstyle 0$};
		\filldraw (-.4,-.5) circle (2pt);
		\node at (.5,-.8){$\scriptstyle b$};
		\draw[<->] (.7,0)--(1.4,0);  
		\end{tikzpicture}
		\begin{tikzpicture}[scale=.5]
		\draw (0,-.5)--(0,.5);
		\node at (0,.8){$\scriptstyle b$};
		\node at (0,-.8){$\scriptstyle b$};
		\end{tikzpicture}
		\begin{tikzpicture}[scale=.5]
		\draw[<->] (-.7,0)--(-1.4,0);
		\draw (0,0)--(0,.5);
		\node at (0,.8){$\scriptstyle b$};
		\draw (0,0)--(.4,-.5); 
		\draw (0,0)--(-.4,-.5);
		\node at (.5,-.8){$\scriptstyle 0$};
		\filldraw (.4,-.5) circle (2pt);
		\node at (-.5,-.8){$\scriptstyle b$};
		\end{tikzpicture}
		\qquad
		\begin{tikzpicture}[scale=.5]
		\draw (0,0)--(0,-.6);
		\draw (0,0)--(.4,.5); 
		\draw (0,0)--(-.4,.5);
		\node at (.5,.7){$\scriptstyle 0$};
		\node at (-.5,.7){$\scriptstyle 0$};
		\filldraw (0,-.6) circle (2pt);
		\node at (.25,-.8){$\scriptstyle 0$};
		\draw[<->] (1.4,0)--(2.2,0);
		
		\node at (.35,.09){$\scriptscriptstyle s$};
		\node at (-.7,.1){$\scriptscriptstyle 1-s$};
		\end{tikzpicture}
		\ \ 
		\begin{tikzpicture}[scale=.5]
		\draw (.8,-.7)--(.8,.7);
		\filldraw (.8,-.6) circle (2pt);
		\node at (1.1,.7){$\scriptstyle 0$};
		\node at (1.1,-.7){$\scriptstyle 0$};
		
		\draw (1.6,-.7)--(1.6,.7);
		\filldraw (1.6,-.6) circle (2pt);
		\node at (1.9,.7){$\scriptstyle 0$};
		\node at (1.9,-.7){$\scriptstyle 0$};
		\end{tikzpicture}	}
\end{center}

We can pass from the original coordinates induced from\!\! \product\! to edge-weight coordinates via the following inductive procedure: Edges containing \counit\ are set to have weights $0$. Edges in $out(G)$ are set to have weight 1. The other edges get their weights from
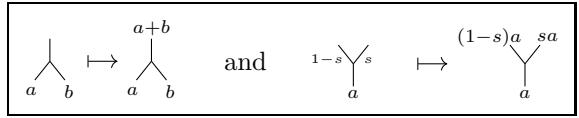
\begin{center}
	\boxed{
		\begin{tikzpicture}[scale=.5]
		\draw (0,0)--(0,.6);
		\draw (0,0)--(.4,-.5); 
		\draw (0,0)--(-.4,-.5);
		\node at (-.5,-.8){$\scriptstyle a$};
		\node at (.5,-.8){$\scriptstyle b$};
		\draw[|->] (1,0)--(1.8,0);  
		\end{tikzpicture}
		\begin{tikzpicture}[scale=.5]
		\draw (0,0)--(0,.6);
		\draw (0,0)--(.4,-.5); 
		\draw (0,0)--(-.4,-.5);
		\node at (-.5,-.8){$\scriptstyle a$};
		\node at (.5,-.8){$\scriptstyle b$};
		\node at (0,.9){$\scriptstyle a+b$};
		\node at (2.5,0){ and };
		\end{tikzpicture}
		\quad
		\begin{tikzpicture}[scale=.5]
		\draw (0,0)--(0,-.6);
		\draw (0,0)--(.4,.5); 
		\draw (0,0)--(-.4,.5);
		\node at (-.7,.1){$\scriptscriptstyle 1-s$};
		\node at (.4,.09){$\scriptscriptstyle s$};
		\node at (0,-.8){$\scriptstyle a$};
		\draw[|->] (1.7,0)--(2.5,0);  
		\end{tikzpicture}
		\begin{tikzpicture}[scale=.5]
		\draw (0,0)--(0,-.6);
		\draw (0,0)--(-.4,.5); \node at (-.95,.7){$\scriptstyle (1-s)a$};
		\draw (0,0)--(.4,.5); \node at (.6,.7){$\scriptstyle sa$};
		\node at (0,-.8){$\scriptstyle a$};
		\end{tikzpicture}
	}	
\end{center}
The passage from edge-weight coordinates to the original coordinates is induced by  
\begin{center}
	\boxed{
		\begin{tikzpicture}[scale=.5]
		\draw (0,0)--(0,.6);
		\draw (0,0)--(.4,-.5); 
		\draw (0,0)--(-.4,-.5);
		\node at (-.5,-.8){$\scriptstyle a$};
		\node at (.5,-.8){$\scriptstyle b$};
		\node at (0,.9){$\scriptstyle a+b$};
		\draw[|->] (1,0)--(1.8,0);  
		\draw (2.5,0)--(2.5,.6);
		\draw (2.5,0)--(2.9,-.5); 
		\draw (2.5,0)--(2.1,-.5);
		\node at (4.5,0){ and };
		\end{tikzpicture}
		\quad
		\begin{tikzpicture}[scale=.5]
		\draw (0,0)--(0,-.6);
		\draw (0,0)--(.4,.5); 
		\draw (0,0)--(-.4,.5);
		\node at (.5,.7){$\scriptstyle c$};
		\node at (-.5,.7){$\scriptstyle b$};
		\node at (0,-.8){$\scriptstyle a$};
		\draw[|->] (1,0)--(1.8,0);  
		
		\draw (3,0)--(3,-.6);
		\draw (3,0)--(3.4,.5); 
		\draw (3,0)--(2.6,.5); 
		\node at (2.5,.1){$\scriptscriptstyle \frac{b}{a}$};
		\node at (3.5,.1){$\scriptscriptstyle \frac{c}{a}$};
		\end{tikzpicture}
	}
\end{center}

\begin{definition} \label{definition: MS}
	Let $\MS$ be the quotient of $\S$ by the involutive, coassociative, associative, commutative and Leibniz relations
	\begin{equation*}
	\boxed
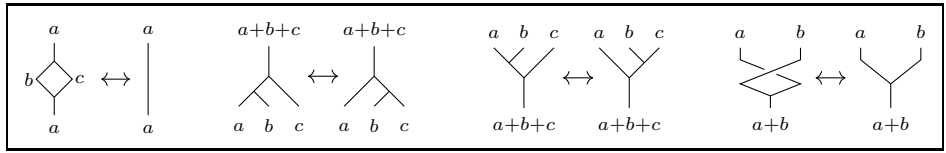
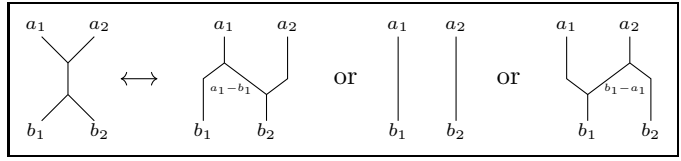
{
		\begin{tikzpicture}[scale=.48]
		\path[draw] (0,0)--(0,.5)--(-.5,1)--(0,1.5)--(0,2);
		\path[draw] (0,.5)--(.5,1)--(0,1.5);
		\node[below] at (0,0){$\scriptstyle a$};
		\node[above] at (0,2){$\scriptstyle a$};
		
		\node at (-.7,1) {$\scriptstyle b$};
		\node at (.7,1) {$\scriptstyle c$};
		
		\draw[<->] (1.3,1)--(2.1,1);
		
		\draw (2.6,2)--(2.6,0);
		
		\node[below] at (2.6,0){$\scriptstyle a$};
		\node[above] at (2.6,2){$\scriptstyle a$};
		\end{tikzpicture}	
		\qquad \
		\begin{tikzpicture}[scale=.4]
		\path[draw] (0,0)--(0,-1)--(-1,-2);
		\draw (0,-1)--(1,-2);
		\draw (-.5,-1.5)--(0,-2);
		
		\node[above] at (0,0){$\scriptstyle a+b+c$};
		\node at (-1,-2.7){$\scriptstyle a$};
		\node at (0,-2.65){$\scriptstyle b$};
		\node at (1,-2.7){$\scriptstyle c$};
		
		\draw[<->] (1.3,-1)--(2.3,-1);
		
		\path[draw] (3.5,0)--(3.5,-1)--(2.5,-2);
		\draw (3.5,-1)--(4.5,-2);
		\draw (4,-1.5)--(3.5,-2);
		
		\node[above] at (3.5,0){$\scriptstyle a+b+c$};
		\node at (2.5,-2.7){$\scriptstyle a$};
		\node at (3.5,-2.65){$\scriptstyle b$};
		\node at (4.5,-2.7){$\scriptstyle c$};
		\end{tikzpicture}
		\qquad \ 
		\begin{tikzpicture}[scale=.4]
		\path[draw] (0,0)--(0,1)--(-1,2);
		\draw (0,1)--(1,2);
		\draw (-.5,1.5)--(0,2);
		
		\node[below] at (0,0){$\scriptstyle a+b+c$};
		\node[above] at (-1,2){$\scriptstyle a$};
		\node[above] at (0,2){$\scriptstyle b$};
		\node[above] at (1,2){$\scriptstyle c$};
		
		\draw[<->] (1.3,1)--(2.3,1);
		
		\path[draw] (3.5,0)--(3.5,1)--(2.5,2);
		\draw (3.5,1)--(4.5,2);
		\draw (4,1.5)--(3.5,2);
		
		\node[below] at (3.5,0){$\scriptstyle a+b+c$};
		\node[above] at (2.5,2){$\scriptstyle a$};
		\node[above] at (3.5,2){$\scriptstyle b$};
		\node[above] at (4.5,2){$\scriptstyle c$};
		\end{tikzpicture}
		\qquad
		\begin{tikzpicture}[scale=.4]
		\path[draw] (0,0)--(0,0.4)--(-1,.8)--(1,1.6)--(1,2);
		\path[draw] (0,0.4)--(1,.8)--(0.2,1.12);
		\path[draw] (-0.2,1.28)--(-1,1.6)--(-1,2);
		
		\node[below] at (0,0){$\scriptstyle a+b$};
		\node[above] at (-1,2){$\scriptstyle a$};
		\node[above] at (1,2){$\scriptstyle b$};
		
		\draw[<->] (1.5,1)--(2.5,1);
		
		\path[draw] (4,0)--(4,0.8)--(3,1.6)--(3,2);
		\path[draw] (4,0.8)--(5,1.6)--(5,2);
		
		\node[below] at (4,0){$\scriptstyle a+b$};
		\node[above] at (3,2){$\scriptstyle a$};
		\node[above] at (5,2){$\scriptstyle b$};
		\end{tikzpicture}
	}
	\end{equation*}
	and 
	\begin{equation*}
	\boxed{
		\begin{tikzpicture}[scale=.7]
		\draw (0,.3)--(0,-.3);
		\draw (0,.3)--(.5,.8);\node at (.6,1){$\scriptstyle a_2$};
		\draw (0,.3)--(-.5,.8);\node at (-.6,1){$\scriptstyle a_1$};
		\draw (0,-.3)--(.5,-.8);\node at (.6,-1){$\scriptstyle b_2$};
		\draw (0,-.3)--(-.5,-.8);\node at (-.6,-1){$\scriptstyle b_1$};
		
		\draw[<->] (1,0)--(1.7,0);
		\end{tikzpicture}
		\ \ \
		\begin{tikzpicture}[scale=.7]
		\draw (-.9,.8)--(-.9,.3)--(-1.3,0)--(-1.3,-.8);
		\draw (-.9,.3)--(-.1,-.3);
		\draw (.3,.8)--(.3,0)--(-.1,-.3)--(-.1,-.8);
		\node at (1.4,0){ or };
		\node at (-.9,1){$\scriptstyle a_1$}; \node at (.3,1){$\scriptstyle a_2$};  		
		\node at (-1.3,-1){$\scriptstyle b_1$}; \node at (-.1,-1){$\scriptstyle b_2$}; 
		\draw (-.77,-.19) node[scale = .8] {$\scriptscriptstyle a_1-b_1$};	
		
		\draw (2.4,.8)--(2.4,-.8);
		\draw (3.5,.8)--(3.5,-.8);
		\node at (2.4,1){$\scriptstyle a_1$}; \node at (3.5,1){$\scriptstyle a_2$};  		
		\node at (2.4,-1){$\scriptstyle b_1$}; \node at (3.5,-1){$\scriptstyle b_2$}; 
		
		\node at (4.5,0){ or };
		\draw (5.6,.8)--(5.6,0)--(6,-.3)--(6,-.8);
		\draw (6,-.3)--(6.8,.3);
		\draw (6.8,.8)--(6.8,.3)--(7.2,0)--(7.2,-.8);
		\node at (5.6,1){$\scriptstyle a_1$}; \node at (6.8,1){$\scriptstyle a_2$};  		
		\node at (6,-1){$\scriptstyle b_1$}; \node at (7.2,-1){$\scriptstyle b_2$};
		\draw (6.72,-.19) node[scale = .8]{$\scriptscriptstyle b_1-a_1$}; 
		\end{tikzpicture}
	}
	\end{equation*}
	depending respectively on if $a_1 > b_1$, $a_1 = b_1$ or $a_1 < b_1$.
\end{definition}

\begin{remark}\label{remark: composition in terms of cutting squares}
	We can express the Leibniz relation in the following alternative way. Consider
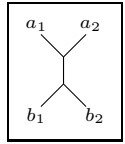
	\begin{equation*}
	\boxed{\begin{tikzpicture}[scale=.6]
		\draw (0,.3)--(0,-.3);
		\draw (0,.3)--(.5,.8);\node at (.6,1){$\scriptstyle a_2$};
		\draw (0,.3)--(-.5,.8);\node at (-.6,1){$\scriptstyle a_1$};
		\draw (0,-.3)--(0,.3);
		\draw (0,-.3)--(.5,-.8);\node at (.7,-1){$\scriptstyle b_2$};
		\draw (0,-.3)--(-.5,-.8);\node at (-.6,-1){$\scriptstyle b_1$};
		\end{tikzpicture}}
	\end{equation*}
	with $a_1 + a_2 = b_1 + b_2$. In $\R^2$ consider the rectangle with opposite vertices at coordinates $(0,0)$ and $(b_1 + b_2, 3)$. Cut along the lines joining $(b_1, 0)$ with $(b_1 , 2)$ and $(a_1 , 3)$ with $(a_1 , 1)$. Deformation retract keeping the vertical coordinate invariant to a $(2,2)$-graph with labelings induced from the plane. Give this $(2,2)$-graph the edge-weight coordinates coming from the width of their corresponding sub-rectangle For example, if $a_1 > b_1$ we have
	
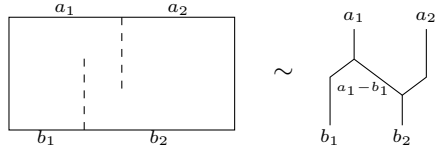
\begin{center}
		\begin{tikzpicture}[scale=.5]
		\draw (0,0) rectangle (6,3); 
		
		\node at (1.5,3.2){$\scriptstyle a_1$};
		\node at (4.5,3.2){$\scriptstyle a_2$};
		\draw[dashed] (3,3)--(3,1);
		
		\node at (1,-.2){$\scriptstyle b_1$};
		\node at (4,-.2){$\scriptstyle b_2$};
		\draw[dashed] (2,0)--(2,2);
		
		\node at (7.3, 1.5) {$\sim$};
		\end{tikzpicture}
		\begin{tikzpicture}[scale=.8]
		\draw (-.9,.8)--(-.9,.3)--(-1.3,0)--(-1.3,-.8);
		\draw (-.9,.3)--(-.1,-.3);
		\draw (.3,.8)--(.3,0)--(-.1,-.3)--(-.1,-.8);
		
		\node at (-.9,1){$\scriptstyle a_1$}; \node at (.3,1){$\scriptstyle a_2$};  		
		\node at (-1.3,-1){$\scriptstyle b_1$}; \node at (-.1,-1){$\scriptstyle b_2$}; 
		\node at (-.75,-.19){$\scriptscriptstyle a_1-b_1$};	
		\end{tikzpicture}
	\end{center}
\end{remark}

\begin{notation} \label{notation: higher valence}
	We will utilize the following diagrammatic simplification
	\begin{center}
		\boxed
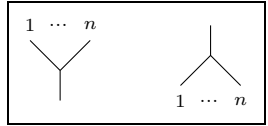
{\begin{tikzpicture}[scale=.4]
			\draw (6,2)--(7,1)--(7,0);
			\draw (8,2)--(7,1);
			\node at (6,2.5){$\scriptstyle 1$};
			\node at (7,2.5){$\scriptstyle \dots$};
			\node at (8,2.5){$\scriptstyle n$};
			
			\draw (11,.5)--(12,1.5)--(12,2.5);
			\draw (13,.5)--(12,1.5);
			\node at (11,0){$\scriptstyle 1$};
			\node at (12,0){$\scriptstyle \dots$};
			\node at (13,0){$\scriptstyle n$};
			\end{tikzpicture}}
	\end{center}
	to represent labeled directed graphs resulting from iterated grafting of the product and coproduct in the left comb order
	\begin{center}
		\boxed
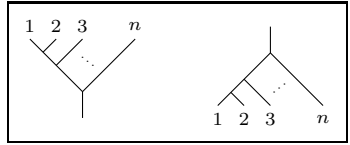
{\begin{tikzpicture}[scale=.35]		
			\node at (-2.15,-.25){};
			\node at (-2.15,3.25){};
			
			\node at (-2,3.5){$\scriptstyle 1$};
			\node at (-1,3.5){$\scriptstyle 2$};
			\node at (0,3.5){$\scriptstyle 3$};
			\node at (2,3.5){$\scriptstyle n$};
			
			\draw (-2,3)--(0,1);
			\draw (2,3)--(0,1)--(0,0);
			\draw (0,3)--(-1,2);
			\draw (-1,3)--(-1.5,2.5);
			\draw (.2,2.3) node[scale= 0.5] {$\ddots$};
			\end{tikzpicture}
			\qquad 
			\begin{tikzpicture}[scale=.35]	
			
			\node at (-2,-3.5){$\scriptstyle 1$};
			\node at (-1,-3.5){$\scriptstyle 2$};
			\node at (0,-3.5){$\scriptstyle 3$};
			\node at (2,-3.5){$\scriptstyle n$};
			
			\draw (-2,-3)--(0,-1);
			\draw (2,-3)--(0,-1)--(0,0);
			\draw (0,-3)--(-1,-2);
			\draw (-1,-3)--(-1.5,-2.5);
			\draw (.2,-2.3) node[scale= 0.5, rotate = 75] {$\ddots$};
			\end{tikzpicture}}
	\end{center}
\end{notation}	

\begin{definition} \label{definition: surjection-like element}
	A canonical $(n,m)$-graph is an $(n,m)$-graph of the form 
	\begin{center}
		\boxed
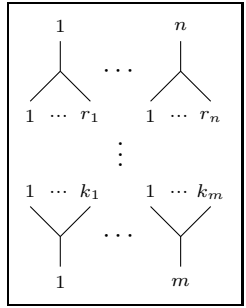
{\begin{tikzpicture}[scale=.4]
			\node at (3,8){$\scriptstyle 1$};
			\draw (2,5.5)--(3,6.5)--(3,7.5);
			\draw (4,5.5)--(3,6.5);
			\node at (2,5){$\scriptstyle 1$};
			\node at (3,5){$\scriptstyle \dots$};
			\node at (4,5){$\scriptstyle r_1$};
			
			\node at (5,6.5){$\cdots$};
			
			\node at (7,8){$\scriptstyle n$};
			\draw (6,5.5)--(7,6.5)--(7,7.5);
			\draw (8,5.5)--(7,6.5);
			\node at (6,5){$\scriptstyle 1$};
			\node at (7,5){$\scriptstyle \dots$};
			\node at (8,5){$\scriptstyle r_n$};
			
			\node at (5,4){$\vdots$};
			
			\node at (3,-.5){$\scriptstyle 1$};
			\draw (2,2)--(3,1)--(3,0);
			\draw (4,2)--(3,1);
			\node at (2,2.5){$\scriptstyle 1$};
			\node at (3,2.5){$\scriptstyle \dots$};
			\node at (4,2.5){$\scriptstyle k_1$};
			
			\node at (5,1){$\cdots$};
			
			\node at (7,-.5){$\scriptstyle m$};
			\draw (6,2)--(7,1)--(7,0);
			\draw (8,2)--(7,1);
			\node at (6,2.5){$\scriptstyle 1$};
			\node at (7,2.5){$\scriptstyle \dots$};
			\node at (8,2.5){$\scriptstyle k_m$};
			\end{tikzpicture}}
	\end{center}
	containing no internal vertices or copies of either\, \counit\ \,or \involution\ and such that for each $i = 1, \dots, m$ the induced map 
	\begin{equation*}
	\{1,\dots,k_i\} \to \bigsqcup\, \{1,\dots,r_1\} <  \cdots < \{1,\dots,r_n\}
	\end{equation*}
	is order preserving.
\end{definition}

\begin{definition}
	Given an element $\gamma$ in $\MS(n,m)$, thought of as an equivalence of weighted $(n,m)$-graphs, we say that an $(n,m)$-graph $\Gamma$ \textbf{supports} it if $\Gamma$ is equal to a representative of $\gamma$ after forgetting its weights.
\end{definition}

\begin{lemma} \label{lemma: unique surjection-like representative}
	For every element in $\MS(n,m)$ with $m>0$ there exists a unique canonical graph supporting it.
\end{lemma}

\begin{proof}
	Consider an element in $\MS(n,m)$ and an $(n,m)$-graph supporting it. We start by getting rid of all occurrences of\, \counit\ . Consider one such strand and follow it up until hitting a vertex, which we must since $m > 0$. If the vertex we encounter is in a subgraph of the form
	\begin{tikzpicture}[scale=.25]
	\draw (0,0)--(0,-.7);
	\draw (0,0)--(.5,.5) ;
	\draw (0,0)--(-.5,.5);
	\draw [fill] (0,-.65) circle [radius=0.1];
	\end{tikzpicture}
	we can replace this with
	\begin{tikzpicture}[scale=.25]
	\draw (0,0)--(0,1.3);
	\draw [fill] (0,0) circle [radius=0.1];
	\draw (.7,0)--(.7,1.3);
	\draw [fill] (.7,0) circle [radius=0.1];
	\end{tikzpicture}
	and continue the excursion up along one of the strands. If alternatively we encounter a vertex contained in a subgraph of one of the following forms
	\begin{tikzpicture}[scale=.25]
	\draw (0,0)--(0,.8);
	\draw (0,0)--(.5,-.5);
	\draw (0,0)--(-.5,-.5);
	\draw [fill] (-.5,-.5) circle [radius=0.1];
	\end{tikzpicture}
	or
	\begin{tikzpicture}[scale=.25]
	\draw (0,0)--(0,.8);
	\draw (0,0)--(.5,-.5);
	\draw (0,0)--(-.5,-.5);
	\draw [fill] (.5,-.5) circle [radius=0.1];
	\end{tikzpicture} 
	we can replace this with\, 
	\begin{tikzpicture}[scale=.25]
	\draw (0,0)--(0,1.4);
	\end{tikzpicture}
	\,and choose another strand\ \counit \ to repeat the process.
	
	We have constructed a $(n,m)$-graph with no copies of\ \ \counit\, supporting our element. We now use the Leibniz relation to ensure that with respect to the direction of the $(n,m)$-graph all vertices belonging to a subgraph of the form\ \coproduct\ appear before vertices belonging to subgraphs of the form \nakedproduct. 
	
	Now we now use coassociativity and associativity to enforce the left comb convention. Using commutativity we reorder the strands of each iterated graftings of \nakedproduct\ so that the order preserving condition is satisfied. We then scan the supporting graph and replace each copy of \involution\ by a copy of \ \identity\, . This construction produces a canonical graph supporting our element. 
	
	In order to show the uniqueness of such canonical graph, we need to prove that the order in which we performed the replacements above does not matter, in the terminology of Gr\"obner bases \cite{dotsenko2010grobner, loday2012algebraic}, that all critical monomials are confluent. For example,
	\begin{center}
		\boxed{
			\begin{tikzpicture}[scale=.26]
			\draw (-1.5,14)--(-3.5,16)--(-3.5,17)--(-2.5,18);
			\draw (-3.5,14)--(-4.5,15)--(-3.5,16)--(-3.5,17)--(-4.5,18);						
			\draw (-5.5,14)--(-4.5,15);
			
			\draw (1.5,14)--(3.5,16)--(3.5,17)--(2.5,18);
			\draw (3.5,14)--(4.5,15)--(3.5,16)--(3.5,17)--(4.5,18);						
			\draw (5.5,14)--(4.5,15);
			
			\draw (-7,7)--(-7,9)--(-6,10)--(-6,11);
			\draw (-7,9)--(-8,10)--(-8,11);
			\draw (-8,7)--(-9,8)--(-9,9)--(-8,10);
			\draw (-10,7)--(-9,8);
			
			\draw (-12,11)--(-12,7); 
			\draw (-14,11)--(-14,8)--(-13,7);
			\draw (-14,8)--(-15,7);
			
			\draw (-17,7)--(-17,9)--(-18,10)--(-18,11);
			\draw (-18,7)--(-19,8)--(-19,9)--(-18,10);
			\draw (-20,7)--(-19,8);
			\draw (-19,9)--(-20,10)--(-20,11);
			
			\draw (-5,4)--(-5,1)--(-4,0);
			\draw (-5,1)--(-6,0);
			\draw (-7,4)--(-7,0);
			
			\draw (-9,0)--(-10,1)--(-10,2)--(-11,3)--(-11,4);
			\draw (-11,0)--(-10,1);
			\draw (-12,0)--(-12,2)--(-11,3);
			\draw (-12,2)--(-13,3)--(-13,4);
			
			\draw (-2,0)--(-2,2)--(-1,3)--(-1,4);
			\draw (0,0)--(0,2)--(-1,3);
			\draw (0,2)--(1,3)--(1,4);			
			\draw (2,0)--(2,2)--(1,3);
			
			\draw (5,4)--(5,1)--(4,0);
			\draw (5,1)--(6,0);
			\draw (7,4)--(7,0);
			
			\draw (9,0)--(10,1)--(10,2)--(11,3)--(11,4);
			\draw (11,0)--(10,1);
			\draw (12,0)--(12,2)--(11,3);
			\draw (12,2)--(13,3)--(13,4);
			
			\draw (7,7)--(7,9)--(6,10)--(6,11);
			\draw (7,9)--(8,10)--(8,11);
			\draw (8,7)--(9,8)--(9,9)--(8,10);
			\draw (10,7)--(9,8);
			
			\draw (12,11)--(12,7); 
			\draw (14,11)--(14,8)--(13,7);
			\draw (14,8)--(15,7);
			
			\draw (17,7)--(17,9)--(18,10)--(18,11);
			\draw (18,7)--(19,8)--(19,9)--(18,10);
			\draw (20,7)--(19,8);
			\draw (19,9)--(20,10)--(20,11);
			
			\draw[myptr] (8,13)--(9,12);	
			\draw[myptr] (9,6)--(8,5);	
			\draw[myptr] (-8,13)--(-9,12);	
			\draw[myptr] (-9,6)--(-8,5);
			\draw[myptr] (0,16.5)--(1,16.5);
			\draw[myptr] (0,16.5)--(-1,16.5);
			
			\node at (10.5,9.5)[scale=.8]{\Large or};
			\node at (15.5,9.5)[scale=.8]{\Large or};
			
			\node at (-10.5,9.5)[scale=.8]{\large or};
			\node at (-15.5,9.5)[scale=.8]{\large or};
			
			\node at (3.5,2.5)[scale=.8]{\large or};
			\node at (8.5,2.5)[scale=.8]{\large or};
			\node at (-8.5,2.5)[scale=.8]{\large or};
			\node at (-3.5,2.5)[scale=.8]{\large or};
			
			\node at (20.5,3){};
			\node at (-20.5,3){};
			\end{tikzpicture}
		}
	\end{center}
	and 
	\begin{center}
		\boxed{
			\begin{tikzpicture}[scale=.3]
			\draw (3,9)--(4,10)--(4,11)--(3,12)--(4,13)--(4,14);
			\draw (5,9)--(4,10)--(4,11)--(5,12)--(4,13);
			
			\draw (9,6)--(9,7)--(8,8)--(8,9)--(9,10)--(9,11);
			\draw (9,7)--(10,8)--(10,9)--(9,10);
			\draw (11,6)--(11,7)--(10,8);
			
			\draw (14,11)--(14,10)--(13,9)--(13,6);
			\draw (14,10)--(15,9)--(15,6);
			
			\draw (17,6)--(17,7)--(18,8)--(18,9)--(19,10)--(19,11);
			\draw (19,6)--(19,7)--(18,8);
			\draw (19,7)--(20,8)--(20,9)--(19,10);
			
			\draw (-2,0)--(-2,1)--(-3,2)--(-2,3)--(-1,4)--(-1,5);
			\draw (-2,1)--(-1,2)--(-2,3);
			\draw (0,0)--(0,3)--(-1,4);
			
			\draw (4,5)--(4,4)--(3,3)--(3,0);
			\draw (4,4)--(5,3)--(5,0);
			
			\draw (8,0)--(8,3)--(9,4)--(9,5);
			\draw (10,0)--(10,1)--(9,2)--(10,3)--(9,4);
			\draw (10,1)--(11,2)--(10,3);
			
			\draw (-1,10)--(-1,9)--(-2,8)--(-2,7);
			\draw (-1,9)--(0,8)--(0,7);
			
			\draw[myptr] (2,6)--(1,7);	
			\draw[myptr] (7,7)--(6,6);	
			\draw[myptr] (6,11)--(7,10);	
			\draw[myptr] (2,11)--(1,10);
			
			\node at (11.7,8.5)[scale=.8]{\large or};
			\node at (16.3,8.5)[scale=.8]{\large or};
			\node at (6.5,3)[scale=.8]{\large or};
			\node at (1.5,3)[scale=.8]{\large or};
			
			\node at (-3,3){};
			\node at (20,3){};
			\end{tikzpicture}		
		}
	\end{center}	
	The other compositions are verified similarly.
\end{proof}

\subsection{Surface realization of $\MS$} 

For any element in $\MS(n,m)$ with $m>0$ we faithfully associate an oriented surface equipped with a CW-structure having a weighted 1-skeleton. 

\begin{construction} \label{construction: arc surface}
	Consider an element in $\MS(n,m)$ with $m>0$ and the canonical $(n,m)$-graph supporting it as constructed in Lemma \ref{lemma: unique surjection-like representative}. By compactifying the open edges of the graph, we introduce $n+m$ new vertices. We glue to each of them both endpoints of an interval and call the resulting $n+m$ circles the incoming and outgoing boundary circles depending on the direction of the $(n,m)$-graph. We make this graph into a ribbon graph, i.e. give each vertex a cyclic order of its incident edges as follows: For the new $n+m$ vertices choose any cyclic order and for all others chose the natural extension of the total order induced from the labeling.	
	
	Consider the surface associated to this ribbon graph.\footnote{The surface associated to a ribbon graph is constructed by attaching a disk to each ribbon loop. A ribbon loop can be described as follows: Choose an edge of the ribbon graph and a direction for that edge. Select from the edges incident to the forward vertex $v$ the one that follows directly after our original edge in the cyclic order associated to $v$. We provide this second edge with the direction that has $v$ as its backward vertex and repeat this process until returning to our original edge in the original direction.} Remove from it the disks attached to the boundary circles and collapse edges satisfying the following conditions: 1) one and only one of their endpoint is in a boundary circle and 2) no other edge incident to their interior endpoint has the same relative direction (towards or away from the vertex).
	
	We refer to the directed and weighted 1-cells of the resulting CW-surface as \textbf{arcs} and notice the original element in $\MS$ can be recovered from them. \vspace*{-8pt}
\end{construction}

\begin{figure} 
	\begin{equation*}
	\boxed{
		\begin{tikzpicture}[scale=.8]
		\draw[thick] (1,5)--(1,2)--(2,1)--(2,0);
		\draw[thick] (1,4)--(-1,2)--(0,1)--(0,0);
		\draw[thick] (1,4)--(2,3)--(1.3,2.3);
		\draw[thick] (.7,1.7)--(0,1);
		
		\draw[|->] (3.2,2.5) to (4,2.5);
		\end{tikzpicture}
		\quad
		\begin{tikzpicture}[scale=.7]
		\draw (0,.5) arc (0:180: 1  and 0.5);
		\draw (0,.5) arc (0:-180: 1 and 0.5);
		\draw[dashed,color=gray] (-2,-4) arc (0:180: 1  and 0.5);
		\draw (-2,-4) arc (0:-180: 1 and 0.5);
		\draw[dashed,color=gray] (2,-4) arc (0:180: 1  and 0.5);
		\draw (2,-4) arc (0:-180: 1 and 0.5);
		\filldraw (1,-4.5) circle (2pt);
		\filldraw (-3,-4.5) circle (2pt);
		\filldraw (-1,0) circle (2pt);
		
		\filldraw (-1,-.7) circle (2pt);
		\draw [thick] (-1,-.7)--(-1,0);
		
		\draw  (-2,.5) to [out=-90, in=90] (-4,-4);
		\draw  (0,.5) to [out=-90, in=90] (2,-4);
		\draw  (-2,-4) to [out=90, in=90] (0,-4);
		
		\draw [thick] (-1,-.7) to [out=-120, in=90] (-3,-4.5);
		\draw [thick] (-1,-.7) to [out=-90, in=90] (1,-4.5);
		\draw [thick] (-1,-.7) to [out=-60, in=190] (.9,-1.6);
		\draw [thick, dashed] (.9,-1.6) to [out=-90, in=20] (-1,-3.42);
		\draw [thick] (-2.5,-2.8) to [out=60, in=170] (-1.15,-3.41);
		
		\filldraw (-2.43,-2.7) circle (2pt);
		
		\draw[|->] (2.5,-1.7) to (3.3,-1.7);
		\end{tikzpicture}
		\quad 
		\begin{tikzpicture}[scale=.7]
		\draw (0,.5) arc (0:180: 1  and 0.5);
		\draw (0,.5) arc (0:-180: 1 and 0.5);
		\draw[dashed,color=gray] (-2,-4) arc (0:180: 1  and 0.5);
		\draw (-2,-4) arc (0:-180: 1 and 0.5);
		\draw[dashed,color=gray] (2,-4) arc (0:180: 1  and 0.5);
		\draw (2,-4) arc (0:-180: 1 and 0.5);
		\filldraw (1,-4.5) circle (2pt);
		\filldraw (-3,-4.5) circle (2pt);
		\filldraw (-1,0) circle (2pt);
		
		\draw  (-2,.5) to [out=-90, in=90] (-4,-4);
		\draw  (0,.5) to [out=-90, in=90] (2,-4);
		\draw  (-2,-4) to [out=90, in=90] (0,-4);
		
		\draw [thick] (-1,0) to [out=-120, in=110] (-3,-4.5);
		\draw [thick] (-1,0) to [out=-90, in=90] (1,-4.5);
		\draw [thick] (-1,0) to [out=-60, in=190] (.9,-1.6);
		\draw [thick, dashed] (.9,-1.6) to [out=-90, in=20] (-1,-3.42);
		\draw [thick] (-1.8,-3) to [out=0, in=170] (-1.15,-3.41);
		\draw [thick] (-1.8,-3) to [out=180, in=40] (-3,-4.5);
		\end{tikzpicture}	
	}
	\end{equation*}
	\caption{Illustrating Construction \ref{construction: arc surface} with omitted weights and labelings}
\end{figure}

\begin{definition}
	The prop $\mathcal{A}$ is defined by pushing forward the CW and prop structures from $\MS$ to the image of Construction \ref{construction: arc surface}.
\end{definition}

\begin{remark} \label{remark: what happends when the weight goes to 0}
	Notice that $\mathcal A(n,0) = \emptyset$ for every $n$. Also, for a family of element in $\mathcal A$ parametrized by the weight of an arc tending to zero, we see that the limit will remove the arc and the topology of the surface will possibly change.
\end{remark}

\subsection{Relations to earlier work}

The reader familiar with \cite{kaufmann03arc} will recognize the elements of $\mathcal A$ as examples of Arc Surfaces. We make the connection more precise with the following

\begin{proposition} \label{theorem: StLGThree and K are isomorphic}
	The operad $U(\mathcal{A})$ is isomorphic to $\mathcal{S}t\mathcal{LGT}ree^1(m)$ as defined by Kaufmann in \cite{kaufmann09dimension}.
\end{proposition}

\begin{proof}
	Comparing with Definition 2.4 in \cite{kaufmann09dimension} and section 4.1 in the same reference, we notice that any element in $U(\mathcal{A})(m)$ corresponds a to quasi-filling element in $\mathcal{LGT}ree(m)^1$ and that any such element arises this way. In the same reference, Corollary 2.2.7 states that any element in $\mathcal{S}t\mathcal{LGT}ree(m)^1$ corresponds to a unique quasi-filling element in $\mathcal{LGT}ree(m)^1$, so we have a bijection between $\mathcal A(m)$ and $\mathcal{S}t\mathcal{LGT}ree(m)^1$. 
	
	Comparing the stabilization process introduced in Definition 2.24 of \cite{kaufmann09dimension} with Remark \ref{remark: what happends when the weight goes to 0} makes this bijection into a cellular isomorphism.
	
	We can use Remark \ref{remark: composition in terms of cutting squares} to describe the composition in $U(\MS)$ and $\mathcal A$ in terms of vertically invariant deformation retractions of cut rectangles:
	
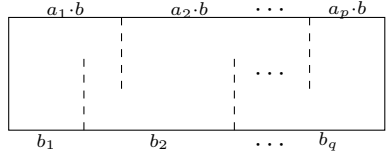
\begin{center}
		\begin{tikzpicture}[scale=.5]
		\draw (0,0) rectangle (10,3); 
		
		\node at (1.5,3.2){$\scriptstyle a_1 \cdot b$};
		\node at (4.8,3.2){$\scriptstyle a_2 \cdot b$};
		\draw[dashed] (3,3)--(3,1);
		
		\node at (1,-.3){$\scriptstyle b_1$};
		\node at (4,-.3){$\scriptstyle b_2$};
		\draw[dashed] (2,0)--(2,2);
		
		\draw[dashed] (6,0)--(6,2);
		\node at (9,3.2){$\scriptstyle a_p \cdot b$};
		\draw[dashed] (8,3)--(8,1);
		\node at (8.5,-.3){$\scriptstyle b_q$};
		
		\node at (7,-.4){$\cdots$};
		\node at (7,1.5){$\cdots$};
		\node at (7,3.2){$\cdots$};
		\end{tikzpicture}
	\end{center}  
	where $a_1 + \cdots + a_q = 1$ and $b_1 + \cdots + b_q = b$. This allows us to recognize the vertical composition in $\mathcal A$ as that of the Arc Surface props \cite{kaufmann03arc}. See for example Section 1.2.2. in \cite{kaufmann09dimension} for the definition of this composition. In particular, this shows the correspondence of the operadic compositions of $U(\mathcal A)$ and $\big\{ \mathcal{S}t\mathcal{LGT}ree^1(m) \big\}_{m \geq 1}$.
\end{proof}
The prop $\chains(\MS)$ in the category of differential graded modules, resulting from applying the functor of cellular chains to the prop $\MS$, inherits a finite presentation with generators
\begin{equation*}
\counit \in \mathrm C_0(\MS)(1,0) \qquad \coproduct \in \mathrm C_0(\MS)(1,2) \qquad \nakedproduct \in \mathrm C_1(\MS)(2,1)
\end{equation*} 
differential 
\begin{equation*}
\partial\ \counit = 0 \qquad \partial\ \coproduct = 0 \qquad \partial\ \nakedproduct =\ 
\begin{tikzpicture}[scale=.22]
\draw (0,0)--(0,1.3);
\draw (.5,0)--(.5,1.3);
\draw [fill] (0,0) circle [radius=0.1];
\draw (.9,.6)--(1.3,.6);
\draw (1.7,0)--(1.7,1.3);
\draw (2.2,0)--(2.2,1.3);
\draw [fill] (2.2,0) circle [radius=0.1];
\end{tikzpicture}
\end{equation*}
and relations
\begin{equation*}
\begin{tikzpicture}[scale=.15]
\draw (1,4)--(1,3)--(0,2)--(0,0);
\draw (1,3)--(2,2)--(2,0);

\draw node[fill,circle, scale=.2] at (0,.1) {};

\draw node[scale=.8] at (3.5,2) {$\scriptstyle -$};

\draw (5,4)--(5,0);

\node at (-.2,4.2){};
\node at(2.2,-.2){};
\node at (7.2,2){;};
\end{tikzpicture}
\begin{tikzpicture}[scale=.15]
\draw (1,4)--(1,3)--(0,2)--(0,0);
\draw (1,3)--(2,2)--(2,0);

\draw node[fill,circle, scale=.2] at (2,.1) {};

\draw node[scale=.8] at (3.5,2) {$\scriptstyle -$};

\draw (5,4)--(5,0);

\node at (-.1,4.2){};
\node at(2.2,-.2){};
\node at (7,2){;\ };
\end{tikzpicture}
\begin{tikzpicture}[scale=.15]
\draw (0,4)--(0,3)--(1,2)--(1,0);
\draw (2,4)--(2,3)--(1,2);

\draw node[fill,circle, scale=.2] at (1,.1) {};

\node at (-.1,4.2){};
\node at(2.2,-.2){};
\node at (4,2){;};
\end{tikzpicture}
\begin{tikzpicture}[scale=.15]
\node at (4.8,4.2){};
\node at (4.8,-.2){};

\draw node[scale=.8] at (5,4){$\scriptstyle 1$};
\draw node[scale=.8] at (7,4){$\scriptstyle 2$};
\draw (5,3)--(5,2)--(6,1)--(6,0);
\draw (7,3)--(7,2)--(6,1)--(6,0);

\draw node[scale=.8] at (8,2){$\scriptstyle -$};

\draw node[scale=.8] at (9,4){$\scriptstyle 2$};
\draw node[scale=.8] at (11,4){$\scriptstyle 1$};
\draw (9,3)--(9,2)--(10,1)--(10,0);
\draw (11,3)--(11,2)--(10,1)--(10,0);

\node at (13,2){;};
\end{tikzpicture} 
\begin{tikzpicture}[scale=.15]
\node at (4.8,4.2){};
\node at (4.8,-.2){};

\draw (7,4)--(7,3)--(6,2);
\draw (5,4)--(5,3)--(7,1)--(7,0);
\draw (8,4)--(8,2)--(7,1)--(7,0);

\draw node[scale=.8] at (9,2){$\scriptstyle -$};

\draw (10,4)--(10,2)--(11,1)--(11,0);
\draw (13,4)--(13,3)--(11,1);
\draw (11,4)--(11,3)--(12,2);

\node at (15,2){;};
\end{tikzpicture} 
\begin{tikzpicture}[scale=.15]
\node at (4.8,-4.2){};
\node at (4.8,.2){};

\draw (7,-4)--(7,-3)--(6,-2);
\draw (5,-4)--(5,-3)--(7,-1)--(7,0);
\draw (8,-4)--(8,-2)--(7,-1)--(7,0);

\draw node[scale=.8] at (9,-2){$\scriptstyle -$};

\draw (10,-4)--(10,-2)--(11,-1)--(11,0);
\draw (13,-4)--(13,-3)--(11,-1);
\draw (11,-4)--(11,-3)--(12,-2);

\node at (15,-2){;};
\end{tikzpicture}\
\begin{tikzpicture}[scale=.15]
\draw (0,4)--(1,3)--(1,1)--(0,0);
\draw (2,4)--(1,3);
\draw (1,1)--(2,0);

\draw node[scale=.8] at (2.7,2){$\scriptstyle -$};

\draw (5,4)--(5,3)--(4,2)--(4,0);
\draw (5,3)--(7,1)--(7,0);
\draw (8,4)--(8,2)--(7,1)--(7,0);

\draw node[scale=.8] at (9,2){$\scriptstyle -$};

\draw (10,4)--(10,2)--(11,1)--(11,0);
\draw (13,4)--(13,3)--(11,1);
\draw (13,3)--(14,2)--(14,0);

\node at (4.8,4.2){};
\node at (14.3,-.2){};
\end{tikzpicture}
\begin{tikzpicture}[scale=.15]
\node at (-1.5,2){;};
\draw (1,4)--(1,3)--(0,2)--(1,1)--(1,0);
\draw (1,3)--(2,2)--(1,1);

\node at (2.2,4.2){};
\node at(2.2,-.2){.};
\end{tikzpicture}
\end{equation*}	

Lemma 17 of \cite{medina2020prop1} shows that the operad associated to $\chains(\MS)$ is isomorphic up to signs to the Surjection operad \cite{mcclure2003multivariable, berger2004combinatorial}. We therefore have the following corollary to Proposition \ref{theorem: StLGThree and K are isomorphic} which was conjectured by Kaufmann in 4.4 of \cite{kaufmann09dimension}.

\begin{corollary}
	The operad obtained by applying the cellular chains to $\mathcal{S}t\mathcal{LGT}ree^1$ is up to signs isomorphic to the Surjection operad.
\end{corollary}

\begin{remark}
	For the $E_2$-suboperad, this was independently established in \cite{kaufmann2017permutahedra}.
\end{remark}

Let us now return to $E_\infty$-structures on simplicial sets. Consider the $\tilde{\S}$-bialgebra structure on the standard simplices as described in Lemma \ref{lemma: action on standard simplices}. Applying the functor of cellular chains, we obtain a natural prop morphism 
\begin{equation} \label{equation: map from chains on first prop to endomorphism of standard simplices}
\chains(\tilde{\S}) \to \End(\chains(\Delta^d))
\end{equation}
where $\End(\chains(\Delta^d))(n,m) = \Hom(\chains(\Delta^d)^{\tensor n}, \chains(\Delta^d)^{\tensor m})$.

As can be seen from Remark  \ref{remark: explicit action of cellular generators on standard simplices}, this map sends the generators of $\chains(\tilde{\S})$ to the following functions, which we describe up to signs:	
\begin{equation} \label{equation: action of chains on S tilde}
\begin{split}
&\coproduct\ [v_0,\dots,v_q] = \sum_{i=0}^q[v_0,\dots,v_i]\otimes[v_i,\dots,v_q], \\
&\ \counit\ [v_0,\dots,v_q] = \begin{cases} 1 & \text{ if } q=0 \\ 0 & \text{ if } q>0,\end{cases} \\
&\nakedproduct\ \big( \left[v_0,\dots,v_p \right] \otimes\left[v_{p+1},\dots,v_q\right] \big) = \begin{cases} 
\left[v_{\pi(0)},\dots,v_{\pi(q)}\right] & \text{ if } i\neq j \text{ implies } v_i\neq v_j \\ \hspace*{1.15cm}
0 & \text{ if not,} \end{cases} \\
&\ \nakedleftcounithomotopy\, [v_0,\dots,v_q] = 0.
\end{split}
\end{equation}

The fact that $\nakedleftcounithomotopy$\ acts trivially serves as a motivation for considering the cellular chains on $\S$ as introduced in Definition \ref{Prop Su}. We have an algebraic presentation of it with generators 
$$\counit \in \chains(\S)(1,0)_0 \hspace*{.6cm} \coproduct \in \chains(\S)(1,2)_0 \hspace*{.6cm} \nakedproduct \in \chains(\S)(2,1)_1$$ 
differential $$\partial\ \counit=0\hspace*{.6cm}\partial\ \coproduct=0\hspace*{.6cm}\partial\ \nakedproduct=\ \boundary$$
and relations $$\productcounit\hspace*{.6cm}\leftcounitality\hspace*{.6cm}\rightcounitality\,.$$ 
We can verify that these relations are satisfied by the assignments in (\ref{equation: action of chains on S tilde}), so we have a factorization
\begin{equation*}
\begin{tikzcd}
\chains(\tilde{\S}) \ar[r] \ar[d] &  \End(\chains(\Delta^d)). \\ \chains(\S) \ar[ur] &
\end{tikzcd}
\end{equation*}
The above morphism $\chains(\S) \to \End(\chains(\Delta^d))$ was introduced and studied in \cite{medina2020prop1} where it was related to the $E_\infty$-structure defined by McClure-Smith and Berger-Fresse \cite{mcclure2003multivariable,berger2004combinatorial}. 

\begin{remark}
	The restriction to biarity $(1,2)$ of this $E_\infty$-bialgebra induces Steenrod's cup-$i$ products as defined in \cite{steenrod47products}. In \cite{medina2018axiomatic}, an axiomatic characterization of these products was given, and in \cite{medina2020globular}, they were used to derive the nerve construction of higher categories as defined in \cite{street1987algebra}. In the present paper, we were able to obtain Steenrod's cup-$i$ products naturally from only four maps associated to the interval. We see this as further evidence of the fundamental nature of Steenrod's cup-$i$ products. See \cite{medina2018persistence} for algorithms using them to incorporate cohomology operations into topological data analysis \cite{carlsson2009data, medina2020giottotda}, and \cite{medina2020cartan, medina2020adem} for their use constructing cochains enforcing the Cartan and Adem relations at the cochain level. Higher arity products, analogous to the cup-i products, defining Steenrod operations at odd primes are defined in \cite{medina2020odd}.
\end{remark}

\bibliography{medina2018cellular}
\bibliographystyle{alpha}
	
\end{document}